\documentclass[12pt]{article}
\usepackage{amssymb,amsmath,amsthm,epsfig}
\usepackage{latexsym, enumerate}
\usepackage{eepic}
\usepackage{epic}
\usepackage{graphicx}
\usepackage{subfigure}
\usepackage{color}
\usepackage{ifpdf}
\usepackage{dsfont}
\usepackage{multirow}
\usepackage{array}
\usepackage{indentfirst}
\theoremstyle{plain}
\newtheorem{lem}{Lemma}[section]
\newtheorem{thm}[lem]{Theorem}

\theoremstyle{definition}

\theoremstyle{remark}
\newtheorem{rem}{Remark}[section]

\begin{document}
\title{ \large\bf Multiscale model reduction method for Bayesian inverse problems of subsurface flow}

\author{
Lijian Jiang\thanks{Institute  of Mathematics, Hunan University, Changsha 410082, China. Email: ljjiang@hnu.edu.cn. Corresponding author}
\and
Na Ou\thanks{College of Mathematics and Econometrics, Hunan University, Changsha 410082, China.}
}

\date{}
\maketitle
\begin{center}{\bf ABSTRACT}
\end{center}
This work presents a model reduction approach to the inverse problem in the application of subsurface flows.  One such an application is
to estimate model's inputs and identify model's parameters.  This is often challenging because the complicated multiscale structures are inherently in the model and the estimated
inputs are parameterized in a high-dimensional space.  We often need to estimate  the probabilistic    distribution of the unknown inputs based on some observations.
Bayesian inference is desirable  for this situation and solving the inverse problem. For the Bayesian inverse problem,  the forward model needs to be repeatedly computed for
a large number of samples to get a stationary chain. This requires  large computational efforts.  To significantly improve the computation efficiency, we use generalized
multiscale finite element method and least-squares stochastic collocation method to construct a reduced computational model.
To avoid  the difficulty  of choosing regularization parameter,  hyperparameters are introduced to build a hierarchical model.   We use truncated  Karhunen-Loeve expansion (KLE)
to reduce the dimension of the parameter spaces and decrease the mixed time of Markov chains.  The techniques of hyperparameter and KLE are incorporated into the model
reduction method.  The reduced model is constructed offline.  Then it is computed very efficiently in the online sampling stage.
This strategy can significantly accelerate  the evaluation of the Markov chain and the resultant  posterior distribution converges fast.
We analyze the convergence for the approximation between  the posterior distribution  by the reduced model and the reference posterior distribution by the full-order model.
A few numerical examples in subsurface flows are carried out to demonstrate
the performance of the presented  model reduction method with application of the Bayesian inverse problem.

\begin{keywords}
 Bayesian inverse problem, GMsFEM, LS-SCM, Subsurface flows
\end{keywords}

\section{Introduction}

Subsurface flow model is a fundamental model in water resources and  applied sciences \cite{jegl07}.  Uncertainties exist inherently in subsurface flow models in heterogeneous porous media.
There are uncertainties coming from the model's inputs and parameters. Because of lack of enough knowledge for geophysical process and measurement noise, we may not know
the model inputs or parameters clearly. The uncertainties can propagate through the model and greatly affect on the prediction of the model. To better predict the model's outputs,
we need to  estimate   the model's inputs and parameters based on some limited observations or measurements. The estimation of the model's inputs such as initial condition, boundary
condition and source location, leads to solving inverse problems.

Inverse problems usually need some  indirect observations. Sparse observations and the uncertainties from forward models' prior information may result in ill-posed inverse problem. The ill-posedness means that no solution exists, multiple solutions may exist, or solutions may not depend  on the data continuously. In practical applications, the inevitable measurement error would increase the challenge of obtaining stable and accurate numerical solutions of the inverse problems.

A classical approach to regularize  inverse problems is through the least squares approach and Tikhonov regularization \cite{hw00, at05}, which leads to the optimization problem: minimize the misfit between observed and predicted outputs in an appropriate norm while penalize unwanted features of the parameters by a regularization term. Point estimates of parameters obtained by this approach would be the best-fit parameters in the sense that the values of parameters  fit the data and honor the regularization penalty term simultaneously.
However, what we are interested in may not only point estimates but also the statistical properties of the parameters.  These can be achieved by
Bayesian inference. In this paper, we resort to  Bayesian inference about the unknown parameters for modeling subsurface flows.

The Bayesian approach \cite{jk06, at05} incorporates uncertainties in observations and prior information by Bayesian rule and gives the posterior probability density of the parameters, which enables us to quantify the  uncertainty in the parameters.  We can use the posterior conditional expectation or maximum a posterior (MAP) to characterize the parameters. It has been shown in \cite{as10} that with some specific prior density, searching the MAP of the posterior measure is equivalent to seeking for the solution of the Tikhonov regularization  problem. Except for the advantage of obtaining the complete statistical description of the interested parameters, using hierarchical model \cite{pc07} in the framework of Bayesian inference enables us to avoid the selection of the regularization parameter, which is very challengeable in  Tikhonov regularization method.

 Although  the posterior density can  be expressed  by something proportional to the production of the likelihood and the prior density,
    it is hard to utilize the expression straightforwardly because of  the nonlinearity of the parameter-to-observation  map and lack of analytical form of the forward model. Instead of analyzing the expression of the posterior, we implement sample-based inference by using Markov chain Monte Carlo (MCMC)  method \cite{bw04, cr13, lj08}. The MCMC approach is often  computationally prohibitive as it requires a large number of forward model simulations during the sampling, especially when  the model is computationally intensive, such as large-scale PDE-based models. In order to accelerate Bayesian inference in the computationally intensive inverse problems, the main attempts include reducing order or searching for surrogates of the forward models \cite{mf10,gr08, mh10}, or seeking more efficient sampling from the posterior \cite{jc12, ye05, tc14, jm12, jw11}. Reduced order models (ROM) refers to projecting a real world system onto a suitable  subspace
    with lower dimension, such that the resulting system is much less computationally demanding  than the original full-order system \cite{cw11, rs13}.

  Subsurface flow models in heterogeneous porous media usually have a wide range of length scales varying from pore scales to field scales. Numerical multiscale methods can efficiently and accurately solve such multiscale models in a coarse grid. Multiscale Finite Element Method (MsFEM) \cite{hwy97} is one of the  multiscale methods and many other multiscale share its similarity \cite{ye09}.
  The basic idea of MsFEM is to incorporate the small-scale information to multiscale basis functions and capture the impact of small-scale features on the coarse-scale
through a variational formulation. One of the most important features for MsFEM is that the multiscale basis functions can  be computed overhead and used repeatedly for the model with different source terms, boundary conditions and the coefficients with similar multiscale structures \cite{hwy97, ye09, jegl07}.
    Recently, a generalized Multiscale Finite Element Method (GMsFEM) \cite{ye11, ye13} has been  developed to solve multiscale models with complex multiscale structures.  GMsFEM has some advantages over the standard MsFEM. For example, the coarse space in GMsFEM is more flexible and the convergence of GMsFEM is independent of the high-contrastness of the multiscales \cite{ye11,ye13}.

In Bayesian inverse problems, the prior uncertainty can be parameterized by random variables and is incorporated into the model.  The model's output depends on the random parameters.
We can use generalized polynomial chaos (gPC)-based stochastic Galerkin methods \cite{dx10} to propagate prior uncertainty through the forward model \cite{ym07, ymm09}. As an alternative to the stochastic Galerkin approach, stochastic collocation \cite{ym09, dx10} requires only a few number of uncoupled deterministic simulations, with no reformulation of the governing equations of the forward model. A sparse grid collocation method using the Smolyak algorithm is presented in \cite{bg07, dx07}, where a stochastic surrogate model is constructed. However, the growing rate of number of collocation nodes required to achieve a good polynomial approximation leads a great challenge in this scheme. To overcome the difficulty, we can assume that the model's output is a stochastic field and
admits a gPC expansion. Then we choose  a set of collocation nodes and use least-squares methods to determine the coefficients of the gPC basis functions. We call the method as
least-squares stochastic collocation method (LS-SCM).  This method shares the same idea as probabilistic collocation method \cite{llz09}. LS-SCM has  the merits from stochastic Galerkin methods and collocation methods.
The recent work \cite{ly12} employs a stochastic collocation algorithm using $l_1$-minimization to construct stochastic sparse models with limited number of nodes, and their strategy has been  applied to the Bayesian approach to handle nonlinear problems \cite{ly15}. Such sparse stochastic collocation methods may give a feasible approach to solve problems in high dimension random spaces.

This work attempts to intensively study  the inverse problem of subsurface flows in porous media.  We will focus on  the case of saturated and confined subsurface flow, which
    is  characterized by a parabolic equation.  Bayesian approach is used to infer  source location,  boundary and initial conditions for the model.  When the target functions
    are infinite dimensional (e.g.,  boundary and initial conditions), we discretize them on a set of  grid points, and then the solutions are sought in a high-dimensional prior space. In order to
    alleviate the difficulty from high dimensionality  of the unknown parameter space, we use a truncated Karhunen-Loeve expansion (KLE) technique to effectively reduce the dimension, which can decrease the mixed time of the Markov chains.
    To accurately capture the multiscale effects of the subsurface model, we apply GMsFEM to construct  a computational surrogate model, which is used to construct   the sensitivity matrix
    for the estimated inputs. To avoid the difficulty of the regularization of the prior term,
    a hierarchical model is used to infer unknown parameters from the prior density.  The posterior distribution of the hyperparameters in the hierarchical model is affected by the number of selected multiscale basis functions used in GMsFEM.
    To avoid the intensive computation for forward model during MCMC sampling, we construct the reduced order model by combing GMsFEM with LS-SCM, which can give a representation for the model response.
       We can use the representation for the repeated  forward model evaluations  at online stage. This can significantly accelerate sampling posterior.
     To assess the approximation by the reduced order model, we analyze the convergence in terms of  Kullback-Leibler divergence. Our numerical analysis shows that the convergence strongly depends the order of gPC and the number of  multiscale basis functions in GMsFEM.  In the paper, we investigate the inversion for multiple inputs (e.g., source location and boundary flux) simultaneously. To efficiently treat multiple inversion, we decompose the model solution into different parts, each of which corresponds to a single input contribution.
       After the  reduced model is constructed,   we can very efficiently  simulate the reduced order model in likelihood procedure. Then   the unknown  inputs  of the model can be estimated by  sampling  the  posterior distribution
       based on the reduced order model.

The outline of the paper is organized as follows. We begin by formulate a subsurface flow model and its inverse problem  in section 2. Section 3  is devoted to the model reduction using  GMsFEM and  LS-SCM.  Some  sampling methods are also presented in the section. In section 4,  we analyze   the approximation between the posterior distribution of the reduced order model and the  posterior distribution of the full-order model.  In Section 5,  we present a few numerical examples to illustrate the performance of proposed method with applications in inverse subsurface flow problems.
Some conclusions and comments are made finally.

\section{Bayesian inference for inverse problems}

We consider a saturated confined  flow model in highly heterogeneous porous media,  which is  described by the following parabolic equation,
  \begin{equation}
  \label{flow-eq}
  \frac{\partial u(x,t)}{\partial t}=\text{div}\bigg(k(x)\nabla u(x,t)\bigg)+f(x),\ \ x\in\Omega, t\in(0,T],
  \end{equation}
subject to an  appropriate boundary condition and initial condition.  Here the coefficient $k(x)$ is a conductivity/permeability field, which may be high contrast and have multiscale structure. The  term $f(x)$ is a source (or sink) term. The solution $u(x,t)$ refers to the water head/pressure. To simplify the function notations, we will suppress the variables $x$ and $t$ in functions when no ambiguity occurs. For practical models,  the model inputs such as boundary/initial  condition and source locations may be not known, and they need to be estimated by some observations or measurements.

In the paper, we use Bayesian inference to estimate the unknown initial/boundary conditions and  the source location for the subsurface flow model by some given noisy measurements of the model response at various sensors. We consider the case of additive noise $e$ with probability density function $\pi(e)$, the measurement data can then be expressed by
  \[
    d = G(z)+e,
  \]
where $z$ is a vector of model parameters or inputs and $G(z)\in\mathbb{R}^{n_d}$ is the model response at measurement sensors, where $n_d$ is the dimension of observations. We assume that $e$ is independent of $Z$, then the conditional probability density for the measurement data $d$ given the unknown $z$, i.e., the likelihood function is given by
  \begin{equation}
  \label{likeli_noise}
    \pi(d|z)= \pi \big(d-G(z)\big).
  \end{equation}

We use Bayesian inference to solve the inverse problem. This approach gives not only  a point estimation but also a probability distribution. This is an advantage of Bayesian method over the standard regularization method. In the Bayesian setting, both $z$ and $d$ are random variables. Then the posterior probability density for $z$ can be derived by the Bayesian rule,
  \begin{equation}
  \label{Bayes}
  \pi(z|d)\propto \pi(d|z)\pi(z),
  \end{equation}
where $\pi(z)$ is the prior distribution with available prior information before the data is observed. The  data enters the Bayesian formulation through the likelihood function $\pi(d|z)$. For the convenience of notation, we will use $\pi^d(z)$ to denote the posterior density $\pi(z|d)$ and $L(z)$ to denote the likelihood function $\pi(d|z)$. Then $(\ref{Bayes})$ can be written as
  \begin{equation}
  \label{like_L}
  \pi^d(z)\propto L(z)\pi(z).
  \end{equation}
Furthermore, if the prior density is conditional to unknown parameter $\mu$, i.e., $\pi(z|\mu)$, the parameter $\mu$ is also a part of the inference problem in the Bayesian framework. In other words, these hyperparameters may be endowed with priors and estimated from data
  \[
  \pi(z, \mu|d)\propto L(z)\pi(z|\mu)\pi(\mu).
  \]
In the paper, we will consider a  hierarchical statistical model  for the  inverse problem with application in subsurface flow.

The vector $e$ is assumed to be independent and identically distributed (i.i.d.) Gaussian random vector with mean zero and standard deviation $\sigma$,
  \[
  e\sim N(0,\sigma^2\mathbb{I}),
  \]
where $\mathbb{I}$ is the identity matrix of size $n_d\times n_d$. Then the likelihood $L(z)$ defined as $(\ref{likeli_noise})$ is given by
  \begin{equation}
  \label{LH}
  L(z) = (2\pi\sigma^2)^{-\frac{n_d}{2}}\exp\bigg(-\frac{\|d-G(z)\|_2^2}{2\sigma^2}\bigg),
  \end{equation}
where $\| \cdot \|_2$ refers to the Euclidean norm. We note that it is not necessary to compute the normalized term in (\ref{like_L}) under most circumstances.

As the posterior distribution of $z$ can be inferred, we can extract the posterior mean or the maximum a posteriori (MAP) of the unknowns.  The MAP estimate is equivalent to the solution of a regularization minimization problem for some specific priors. However, the analytical expression of the posterior distribution is generally unavailable and the high dimension integration involved in posterior expectation is a great challenge. Markov chain Monte Carlo (MCMC) methods are a class of algorithms for sampling from a probability distribution based on constructing a Markov chain that has the desired distribution as its equilibrium distribution, and we can use the method to explore the posterior state space of the unknowns. When a set of independent samples $\{z^{(j)}\}_{j=1}^N$ successively drawn from the posterior, the conditional posterior expectation can be approximated by
  \[
    \mathbb{E}[z|d]=\frac{1}{N}\sum_{j=1}^{N} z^{(j)}.
  \]
The marginal posterior mode (MPM) of the unknowns can also be computed by the $N$ samples. To build a Markov chain with the posterior as its equilibrium distribution by the MCMC method, we need to call a large numbers of  deterministic forward solvers, which may be computationally expensive and inefficient.  For practical  subsurface flow model, it may be infeasible
to resolve all scales in very fine grid.
To overcome the difficulties from MCMC sampling and the multiscale features,  we apply  GMsFEM  and LS-SCM to construct a surrogate model for computation. The surrogate model
is defined in a coarse grid and its uncertainty is parameterized in a low dimensional space.
 This can significantly reduce the computation cost in the process of  the likelihood computation defined in (\ref{LH}).

%%%%%%%%%%%%%%%%%%%%%%%%%%%%%%%%%%%%%%%%%%%%%%%%%%%%%%%%%%%%%%%%%%%%%%

\section{Model reduction based on GMsFEM and LS-SCM}

In this section, we use GMsFEM and LS-SCM to build a reduced order model for the inverse problem of subsurface  flow.

\subsection{GMsFEM}

 GMsFEM can achieve efficient forward model simulation and provide an accurate  approximation for  the solution of  multiscale problems.
  In this section, we follow the idea of  GMsFEM \cite{ye11, ye13} and apply it  to the subsurface flow equation (\ref{flow-eq}).
  For GMsFEM, we need to pre-compute a set of multiscale basis functions.  To this end,
  We solve the following local eigenvalue problem on each coarse block $\omega_i$,
  \begin{eqnarray}
  \label{MS-basis}
   \begin{cases}
 & -\text{div}(k\nabla \psi_{il}) =\lambda k \psi_{il},\ \ \text{in}\ \omega_i\\
 & k\nabla \psi_{il}\cdot \vec{n} = 0,\ \ \text{on}\ \partial \omega_i.
 \end{cases}
  \end{eqnarray}
This can be discretized as
  \[
   A\psi_{il}=\lambda S\psi_{il},
  \]
where
  \[
  A=[a_{mn}]=\int_{\omega_i} k\nabla l_n\nabla l_m, \quad S=[s_{mn}]=\int_{\omega_i} k l_nl_m,
  \]
where $l_n$ denotes the basis functions in fine grid. We take the first $M_i$ eigenfunctions corresponding to the dominant eigenvalues for each coarse neighborhood $\omega_i$ (see Figure \ref{coarse-cell}),
 $i=1,2,\cdots,N_H$,  where $N_H$ is the number of coarse nodes.  For each coarse element $K\in\omega_i$, let $\chi_i$ be the solution to the equation
  \[
    \left\{
  \begin{aligned}
  -\text{div}(k\nabla \chi_i)&=0,\ \ K\in\omega_i\\
  \chi_i&=g_i,\ \ \text{on}\ \partial K,
  \end{aligned}
     \right.
  \]
where $g_i$ is a  linear hat function.  The relationship between a coarse neighborhood and its coarse elements is illustrated in  Figure \ref{coarse-cell}.
 Thus $\{\chi_i\}_{i=1}^{N_H}$ form a set of partition of unity functions associated with the open cover $\{\omega_i\}_{i=1}^{N_H}$ of $\Omega$.
\begin{figure}
\label{coarse-cell}
  \centering
  \includegraphics[width=0.6\textwidth]{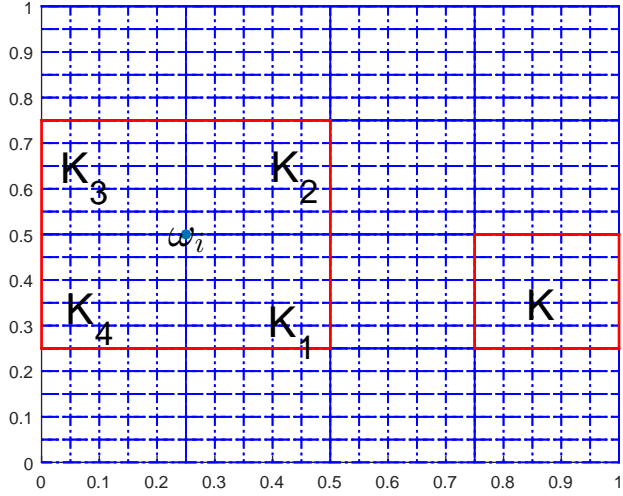}
  \caption{Illustration of a coarse neighborhood and a coarse element}
\end{figure}
Then we  multiply the partition of unity functions by the eigenfunctions to construct GMsFE space,
  \[
  V_H=\text{span}\{\Psi_{il}: \Psi_{il}=\chi_i\psi_{il}: 1\leq i\leq N_H \quad \text{and}\quad 1\leq l\leq M_i\}.
  \]
We use  a single index  for  the multiscale basis function set  $\{\Psi_{il}\}$ and place them    in  the following matrix
  \[
  R=[\Psi_1, \Psi_2,\cdots,\Psi_{M_v}],
  \]
where $M_v=\sum_{i=1}^{N_H} M_i$ denotes the total number of multiscale basis  functions.
 We note that once the matrix $R$ constructed, it can be repeatedly used for simulation.

In the paper, the backward  Euler scheme is used for temporal discretization.  Let $U^n$ be  the solution  at the $n-$th time level $t_n=n\Delta t$, where $\Delta t$ is the time step. Then we have the weak formulation for the
parabolic equation (\ref{flow-eq}),
  \[
    \left\{
  \begin{aligned}
  \bigg(\frac{U^n-U^{n-1}}{\Delta t}, v\bigg)+a(U^n, v)&=(f(t_n), v),\quad \forall v\in V_H\\
  (U^0,v)&=(u(x,0), v),\quad \forall v\in V_H,
  \end{aligned}
     \right.
  \]
where $( , )$ denotes the usual $L_2$ inner product and
  \[
    a(u,v)=\int k\nabla u\nabla vdx.
  \]
We assume the $U^n$ has the approximation
  \[
  U^n=\sum_{j=1}^{M_v} \alpha_{Hj}^n \Psi_j(x),
  \]
where the subscript $H$ denotes the GMsFEM solution on coarse grid. Let
\[
\alpha_H^n=(\alpha_{H1}^n,\alpha_{H2}^n,\cdots,\alpha_{HM_v}^n)^T.
\]
Then   for $k=1,\cdots,M_v$,
  \begin{equation}
  \label{c-eq}
  \sum_{j=1}^{M_v} \alpha_{Hj}^n(\Psi_j, \Psi_k)+\Delta t\sum_{j=1}^{M_v}\alpha_{Hj}^n a(\Psi_j, \Psi_k)=\sum_{j=1}^{M_v} \alpha_{Hj}^{n-1} (\Psi_j, \Psi_k)+\Delta t(f^n, \Psi_k).
  \end{equation}
 Let  $B$, $K$ and $F$ be the mass, stiffness matrices and load vector using FEM basis function in fine grid,  respectively. Then equation (\ref{c-eq}) gives
   the following algebraic system,
  \[
   R^TBR\alpha_H^n+\Delta tR^TKR\alpha_H^n=R^TBR\alpha_H^{n-1}+\Delta tR^TF,
  \]
 If we denote
 \[\tilde B=R^TBR,  \quad \quad \tilde K=R^TKR,\]
  then $\alpha_H^n$ can be calculated by the iteration
  \begin{equation}
  \label{iteration_c}
  \alpha_H^n=(\tilde B+\Delta t\tilde K)^{-1}(\tilde B\alpha_H^{n-1}+\Delta tR^TF).
  \end{equation}
By using the multiscale basis functions,  the solution in fine grid can be obtained by the transformation
  \[
  \alpha_h^n=R\alpha_H^n.
  \]

We note that when GMsFEM is not applied, the full model solution is obtained by the iteration
  \begin{equation}
  \label{iteration_f}
  \alpha_h^n=(B+\Delta tK)^{-1}(B\alpha_h^{n-1}+\Delta tF).
  \end{equation}
Compared  $(\ref{iteration_c})$ with $(\ref{iteration_f})$, it can be seen that  the size of $\tilde K$ and $\tilde B$ are $M_v \times M_v$, but the size of $K$ and $B$ are $N_h \times N_h$  ($M_v\ll N_h$).
 Thus a much smaller system is solved in GMsFEM.  The matrix $R$ for multiscale basis functions is computed overhead and it can be repeatedly used for all time levels.
 This significantly improves the efficiency for forward model simulations.

\subsection{Stochastic collocation via least-squares method}
\label{SCM}
 Stochastic collocation method is an efficient approach  to approximate the solution of PDEs with random inputs. In the paper, the stochastic collocation method
 is based on generalized polynomial chaos (gPC) and least-squares method.   The approximation solution  can be represented by gPC expansion using the stochastic collocation method.
 We use the stochastic collocation method to solve the forward model.
 With the established gPC expansion of the approximation of the forward model,  the evaluation of the likelihood function $L(z)$ in  MCMC sampling can be significantly accelerated.

We denote the random  parameters as $Z=(Z_1, \cdots, Z_{n_z})$, and assume that each random variable $Z_i$ has a prior probability density function $\pi_i(z_i): \Gamma_i \rightarrow \mathbb{R}$, for $i=1,\cdots, n_z$, where $\Gamma_i$ is the support of $Z_i$. Then the joint prior density function of $Z$ is
  \[
  \pi(z)=\prod_{i=1}^{n_z} \pi_i(z_i),
  \]
and its support has the form
  \[
  \Gamma :=\prod_{i=1}^{n_z} \Gamma_i \in R^{n_z}.
  \]
 If the prior of part of our parameters is bounded, e.g., $\Gamma:=\prod_{i=1}^{n_z} [-1, 1]^{n_z}$,
 we can use  Legendre orthogonal polynomials as the basis functions to construct approximations of the forward model solution.

Without loss of generality, we describe the gPC approximation to the forward model for $n_d=1$. Let $i=(i_1,\cdots,i_{n_z})$ $\in$ $N_0^{n_z}$ be a multi-index with $|i|=i_1+\cdots+i_{n_z}$, and let $N\geq0$ be an integer. The $N$th-degree gPC expansion of $G(Z)$ is defined as
  \begin{equation}
  \label{gpc_t}
    G_N(Z)=\sum_{i=1}^P c_i \Phi_i(Z), \quad \quad P=\frac{(N+n_z)!}{N!n_z!},
  \end{equation}
where
  \begin{equation}
  \label{coef_t}
  c_i=\mathbb{E}[G(Z)\Phi_i(Z)]=\int G(z)\Phi_i(z)\pi(z)dz,
  \end{equation}
are the expansion coefficients, $\mathbb{E}$ is the expectation operator, and $\Phi_i(Z)$ are the basis functions defined as
  \[
  \Phi_i(Z)=\phi_{i_1}(Z_1)\cdots\phi_{i_{n_z}}(Z_{n_z}), \quad 0\leq|i|\leq N,
  \]
where $\phi_m(Z_k)$ is the $m$th-degree one-dimensional orthogonal polynomial having been normalised in the $Z_k$ direction, which satisfies
  \[
  \mathbb{E}_k[\phi_m(Z_k)\phi_n(Z_k)]=\int \phi_m(z_k)\phi_n(z_k)\pi_k(z_k)dz_k=\delta_{m,n}, \quad 0\leq m,n\leq N.
  \]
Thus, $\{\Phi_i(Z)\}$ are $n_z$-variate orthonormal polynomials of degree up to $N$ satisfying
  \begin{equation}\label{ortho-basis}
  \mathbb{E}[\Phi_i(Z)\Phi_j(Z)]=\int \Phi_i(z)\Phi_j(z)\pi(z)dz=\delta_{i,j}, \quad 0\leq |i|,|j|\leq N,
  \end{equation}
where $\delta_{i,j}=\prod_{k=1}^{n_z} \delta_{{i_k},{j_k}}$. Following \cite{ym09}, the gPC expansion $(\ref{gpc_t})$ converges to $G$ as
  \begin{equation}
  \label{gpc_error}
  \|G(z)-G_N(z)\|_{L_{\pi_z}^2}=\bigg(\int \big(G(z)-G_N(z)\big)^2\pi(z)dz\bigg)^{1/2}\leq CN^{-p},
  \end{equation}
where $C$ is a constant independent of $N$, and $p>0$ depends on the smoothness of $G$.

In the stochastic collocation method, we first choose a set of collocation nodes $\{z^{(i)}\}_{i=1}^Q\in\Gamma$, where $Q\geq1$ is the number of nodes. Then  for each $i=1,\cdots,Q$, we solve a deterministic problem at the node $z^{(i)}$ to obtain
  \[
  G(z^{(i)})=g\circ u(x,t;z^{(i)}),
  \]
where $g: \mathbb{R}^{n_u}\rightarrow \mathbb{R}$ is a state  function. After the pairings $\big(z^{(i)}, G(z^{(i)})\big)$ ($ i=1,\cdots,Q$) being obtained, we are able to construct a good approximation of $G_N(z)$, such that $G_N(z^{(i)})=G(z^{(i)})$ for all $i=1,\cdots,Q$. Thus, we need to solve $Q$ deterministic problems.
In the paper we use least-squares method to obtain the coefficient $c$ in $(\ref{coef_t})$.

Let $\{z^{(i)}\}_{i=1}^Q$ be the set of i.i.d. samples for $Z$ and $\{G(z^{(i)})\}_{i=1}^Q$ the corresponding realizations of the stochastic function $G(Z)$.
Let
\[
c=(c_1,\cdots,c_P)^T\in \mathbb{R}^P,  \quad  b=\bigg(G(z^{(1)}),\cdots,G(z^{(Q)})\bigg)^T\in \mathbb{R}^Q.
\]
If we  set the condition $G_N(z^{(i)})=G(z^{(i)}), i=1,\cdots,Q$, then  the following equation holds,
  \begin{equation}
  \label{ls}
  Vc=b,
  \end{equation}
where $V\in \mathbb{R}^{Q\times P}$ is the matrix with the entries
  \[
  V_{ij}=\Phi_j(z^{(i)}), \quad i=1,\cdots,Q,\quad j=1,\cdots,P.
  \]
 The number $Q$ of samples is taken  according to the following rule \cite{mc15},
  \begin{equation}
  \label{ls_rule}
  Q=an_z(N+1)^2,\quad a\geq 1,
  \end{equation}
to get an accurate least-squares solution of $(\ref{ls})$.  We set $a=3$ in this paper. Hence, we get the normal equation
  \[
  V^TVc=V^Tb,
  \]
and the approximation coefficient
  \begin{equation}
  \label{coef_ls}
  \tilde c=(V^TV)^{-1}V^Tb.
  \end{equation}
Thus, we construct the $N$th-order gPC approximation by
  \begin{equation}
  \label{gpc_ls}
    \tilde G_N(Z)=\sum_{m=1}^P \tilde c_i \Phi_i(Z).
  \end{equation}

When constructing the gPC approximation, the forward model will be solved $Q$ times to obtain the sampling vector $b$. It can be seen from $(\ref{ls_rule})$ that the number of samples $Q$ is quadratic $N+1$, as we increase the order of the gPC expansion to pursue accuracy of the approximation, the times we need to solve the deterministic forward model will increase at much a magnitude. We  efficiently solve the problem using GMsFEM. Then we replace the sample vector $b$ by $b_i^M=G^M(z^{(i)})$, where $G^M(z)$ represents the GMsFEM solution at the sensor. By combing GMsFEM with LS-SCM,  we can have the $N$th-order gPC approximation
  \begin{equation}
  \label{gpc_gms}
    \tilde G_N^M(Z)=\sum_{m=1}^P \tilde c_i^M \Phi_i(Z).
  \end{equation}
 The number of multiscale basis functions on each coarse neighborhood would have effect on the accuracy of the approximation. When combining the stochastic collocation method with GMsFEM, the accuracy of the resultant surrogate model will be effected both by the order of the gPC expansion and the number of multiscale basis functions.

\subsection{Metropolis-Hastings algorithm}

 As we noted before, Markov chains are constructed for the exploring of posterior state space, and the Metropolis-Hastings (MH) algorithm is one of the extensively used algorithm to build Markov chains that converge to the posterior distribution of estimated parameters. Repeated call of the forward model is required for the MH algorithm.

The surrogate model constructed by combing GMsFEM with LS-SCM can  accelerate MCMC sampling.  The GMsFE model order reduction method is used  to build the sample vector $b$ and  we have the surrogate model $(\ref{gpc_gms})$, which is used to approximate the posterior distribution and then to explore the posterior state.

Generally, the MH algorithm starts from a random initial value, a Markov chain generates trial moves from the current state $z^{(j)}$ to a new state $z^*$ by acceptance probability:
  \[
    \alpha(z^{(j)},z^*)=\text{min}\{1, \frac{\pi(z^*)q(z^{(j)}|z^*)}{\pi(z^{(j)})q(z^*|z^{(j)})}\},
  \]
where $\pi$ is the target distribution and $q$ is the proposal distribution, of which the scale and orientation will affect the efficiency of the MH algorithm. When the proposal distribution is too wide, many candidate points may be rejected. This will lead to  long mixing time for the chain and slow convergence to the target distribution. On the other hand, when the proposal distribution is too narrow, a high acceptance rate may  make the moved distance so small that a large number of updates will be required to converge to the target distribution. The choice of the proposal distribution is crucial in determining the practical applicability of MCMC simulation in many fields of study.

In the paper, we take $q(z^*|z^{(j)})=U(z^{(j)}-\varepsilon, z^{(j)}+\varepsilon)$, where $\varepsilon$ is the scale of the random walk. The proposal distribution $q$ does not have to be symmetric. The Gibbs sampler is a special kind of Metropolis-Hastings algorithm, in which the proposal distribution is full conditional distributions and the acceptance probability is identically one. For a $n_z-$dimensinal random vector $z$, if the full conditional probability density function of each component is attainable, the Gibbs method can be used to speed up the chain convergence. Let $\pi^{d,M}$ be the  posterior distribution constructed by using GMsFEM, i.e.,
 \[
  \pi^{d,M}(z)\propto \exp\bigg(-\frac{\|d-G^M(z)\|_2^2}{2\sigma^2}\bigg)\pi(z),
 \]
and $\pi_N^{d,M}$ the approximate posterior distribution constructed by combing GMsFEM with LS-SCM, i.e.,
 \[
  \pi^{d,M}_N(z)\propto \exp\bigg(-\frac{\|d-G_N^M(z)\|_2^2}{2\sigma^2}\bigg)\pi(z).
 \]
 Table \ref{MH-Gibbs} shows the steps  of MH algorithm and Gibbs sampling method when GMsFEM and LS-SCM are used in the computation.
\begin{table}
  \centering
  \caption{MH algorithm and Gibbs sampling method}\label{MH-Gibbs}
  \begin{tabular}{l}
    \hline
  \bf{Algorithm 1}: The MH algorithm \rule{0pt}{1cm}\\
  1. Given $z^{(k)}$, draw $z^*\sim q(\cdot|z^{(k)})$;\rule{0pt}{0.5cm}\\
  2. Calculate the acceptance probability \\
     $\alpha(z^{(j)},z^*)=\text{min}\{1, \frac{\tilde \pi_N^{d,M}(z^*)}{\tilde \pi_N^{d,M}(z^{(j)})}\}$
  \\
  3. With probability $\alpha$, accept and set $z^{(k+1)}=z^*$, otherwise set $z^{(k+1)}=z^{(k)}$.\rule{0pt}{0.5cm}\\
  \bf{Algorithm 2}: The Gibbs sampling algorithm \rule{0pt}{1cm}\\
  1. Initialize $z^{(0)}$; \rule{0pt}{0.5cm}\\
  2. For $j=1:N$ \rule{0pt}{0.5cm}\\
  $z^{(j)}$ can be sampled by \rule{0pt}{0.5cm}\\
  $\bullet$ sample $z_1^{(j)}$ $\sim$ $\pi^{d,M}(z_1|z_2^{(j-1)}, z_3^{(j-1)}, \cdots, z_{n_z}^{(j-1)})$ \rule{0pt}{0.5cm}\\
  $\bullet$ sample $z_2^{(j)}$ $\sim$ $\pi^{d,M}(z_2|z_1^{(j)}, z_3^{(j-1)}, \cdots, z_{n_z}^{(j-1)})$ \rule{0pt}{0.5cm}\\
  $\bullet$ \vdots \rule{0pt}{0.5cm}\\
  $\bullet$ sample $z_{n_z}^{(j)}$ $\sim$ $\pi^{d,M}(z_1|z_1^{(j)}, z_2^{(j)}, \cdots, z_{{n_z}-1}^{(j)})$ \rule{0pt}{0.5cm}\\
    end for \rule{0pt}{0.5cm}\\
    \hline
  \end{tabular}
\end{table}

%%%%%%%%%%%%%%%%%%%%%%%%%%%%%%%%%%%%%%

\section{Convergence analysis}

  To study the convergence of posterior using the reduced order method, we use Kullback-Leibler (KL) divergence \cite{ym09} to quantify the difference between the exact posterior and the approximated posterior.
  For probability density functions $\pi_1(z)$ and $\pi_2(z)$, KL divergence  is defined by
  \[
  D_{KL}(\pi_1||\pi_2)=\int \pi_1(z) \log \frac{\pi_1(z)}{\pi_2(z)}dz.
  \]
  $D_{KL}$ measures the difference between two probability distributions and is nonnegative.  It is vanished if and only if $\pi_1=\pi_2$.

\begin{lem}\cite{ly15}\label{lemm}
Suppose the functions $G$ and $\tilde G_N^M$ are under some assumption, and the observational error has an i.i.d. Gaussian distribution. If  the prior of Z is uniform, then the approximation posterior $\tilde \pi_N^{d,M}$ and the true posterior density $\pi^d$ are close with respect to the Kullback-Leibler distance, i.e.,  there is a constant $C$, independent of $N$, such that
  \[
  D_{KL}(\tilde \pi_N^{d,M}\|\pi^d)\leq C\sum_{i=1}^{n_d} \|G_i(z)-\tilde G_{N,i}^M(z)\|_{L_{\pi_z}^2}^2.
  \]
\end{lem}

\begin{thm}
The approximation $\tilde G_N(z)$ obtained by LS-SCM converges to $G_N(z)$, i.e.,
  \[
  \|G_{N,i}(z)-\tilde G_{N,i}(z)\|_{L_{\pi_z}^2}\rightarrow 0, \quad \text{for} \quad i=1,\cdots,n_d,
  \]
as $Q\rightarrow\infty$.
\end{thm}
\begin{proof}
Note that the coefficient in $(\ref{gpc_t})$ is calculated by $(\ref{coef_t})$, which can be written in the matrix form
  \begin{equation}\label{best-fit}
  Xc=r,
  \end{equation}
where the entries are given by
  \begin{eqnarray*}
  X_{i,j} &=& \int \Phi_i \Phi_j\pi(z)dz,\\
  r_j &=& \int G(z)\Phi_j\pi(z)dz.
  \end{eqnarray*}
  Due to the orthogonality of basis $\{\Phi_i(Z)\}$ defined in $(\ref{ortho-basis})$, $X$ is an identity matrix. In the paper, we actually use numerical integration to approximate the matrix $X$ and vector $r$.  The coefficient $\tilde c$ computed by $(\ref{coef_ls})$ satisfies
  \[
  V^TV\tilde c=V^Tb,
  \]
By multiplying  $\frac{1}{Q}$ at both sides of the above equation, we have
  \begin{equation}\label{Q-normaleq}
    \frac{1}{Q}V^TV\tilde c=\frac{1}{Q}V^Tb.
  \end{equation}
It can be seen that $\frac{1}{Q}V^TV$ and $\frac{1}{Q}V^Tb$ are the approximate of $X$ and $r$ in the sense of Monte Carlo integration, respectively. Thus, as $Q\rightarrow\infty$,
  \begin{eqnarray*}
  (\frac{1}{Q}V^TV)_{i,j} &\rightarrow& X_{i,j} \\
  (\frac{1}{Q}V^Tb)_{j} &\rightarrow& r_j, \ \ \text{for} \quad 1\leq i ,j \leq P.
  \end{eqnarray*}
Because  the approximation  error in Monte carlo integration is $O(\frac{1}{\sqrt Q})$. Then we can have the  estimation
  \[
  \|r-\frac{1}{Q}V^Tb\|_2\lesssim \sqrt \frac{P}{Q}.
  \]
On the other hand,  equation $(\ref{best-fit})$ and $(\ref{Q-normaleq})$ implies that
  \begin{eqnarray*}
   \|r-\frac{1}{Q}V^Tb\|_2&=&\|Xc-\frac{1}{Q}V^TV\tilde c\|_2\\
   &\geq& \|X(c-\tilde c)\|_2-\|(X-\frac{1}{Q}V^TV)\tilde c\|_2\\
   &=& \|c-\tilde c\|_2-\|(X-\frac{1}{Q}V^TV)\tilde c\|_2.
  \end{eqnarray*}
Because
  \[
  \|X-\frac{1}{Q}V^TV\|_2\leq \|X-\frac{1}{Q}V^TV\|_F\|\tilde c\|_2,
  \]
where $\| \cdot \|_F$ is the Fubini norm,  we have
  \[
  \|X-\frac{1}{Q}V^TV\|_2\lesssim \sqrt \frac{P^2}{Q}.
  \]
Consequently, it follows that
  \[
  \|c-\tilde c\|_2\lesssim \sqrt \frac{P}{Q}+\sqrt \frac{P^2}{Q}.
  \]
By  $(\ref{gpc_t})$ and $(\ref{gpc_ls})$, we have
  \begin{eqnarray*}
  \|G_{N,i}(z)-\tilde G_{N,i}(z)\|_{L_{\pi_z}^2} &=& \bigg(\int [\sum_{j=1}^P (c_j-\tilde c_j)\Phi_j]^2\pi(z)dz\bigg)^{1/2} \\
  &=& \|c-\tilde c\|_2\\
  &\lesssim & \sqrt \frac{P}{Q}+\sqrt \frac{P^2}{Q}.
  \end{eqnarray*}
As $Q\rightarrow\infty$, it will tends to be zero, i.e.,
  \[
  \lim_{Q\rightarrow\infty}\|G_{N,i}(z)-\tilde G_{N,i}(z)\|_{L_{\pi_z}^2}=0.
  \]
\end{proof}
\begin{thm}\label{G-GM}
Let the model order reduction error from GMsFEM be given by
  \[
  \|b-b^M\|_2 \lesssim \mathcal{E}_{ms},
  \]
where $b=(G_i(z^{(1)}), \cdots, G_i(z^{(Q)}))^T$ and $b^M=(G_i^M(z^{(1)}), \cdots, G_i^M(z^{(Q)}))^T$. Then
\[
  D_{KL}(\tilde \pi_N^{d,M}\|\pi^d)\lesssim n_d \bigg[ O(N^{-2\alpha})+O(Q^{-1})+\mathcal{E}_{ms}^2\bigg].
  \]
\end{thm}
\begin{proof}
By $(\ref{gpc_ls})$ and $(\ref{gpc_gms})$, we have
  \begin{eqnarray*}
  \|\tilde G_{N,i}(z)-\tilde G_{N,i}^M(z)\|_{L_{\pi_z}^2} &=& \bigg(\int [\sum_{j=1}^P (\tilde c_j-\tilde c_j^M)\Phi_j]^2\pi(z)dz\bigg)^{1/2} \\
   &=& \|\tilde c-\tilde c^M\|_2\\
   &=& \|(V^TV)^{-1}V^T(b-b^M)\|_2\\
   &\leq& C_2\|b-b^M\|_2\\
   &\lesssim& \mathcal{E}_{ms}.
  \end{eqnarray*}
The triangle inequality gives
  \begin{eqnarray*}
  \|G_i(z)-\tilde G_{N,i}^M(z)\|_{L_{\pi_z}^2} &\leq& \|G_i(z)-G_{N,i}(z)\|_{L_{\pi_z}^2}+\|G_{N,i}(z)-\tilde G_{N,i}(z)\|_{L_{\pi_z}^2}\\
  &+&\|\tilde G_{N,i}(z)-\tilde G_{N,i}^M(z)\|_{L_{\pi_z}^2} \\
  &\lesssim& O(N^{-\alpha})+O(Q^{-\frac{1}{2}})+\mathcal{E}_{ms},
  \end{eqnarray*}
By  Lemma $\ref{lemm}$, we immediately  have
  \[
  D_{KL}(\tilde \pi_N^{d,M}\|\pi^d)\lesssim n_d \bigg[ O(N^{-2\alpha})+O(Q^{-1})+\mathcal{E}_{ms}^2\bigg].
  \]
\end{proof}

\begin{rem}
If we take the first $M$ dominant eigenfunctions in (\ref{MS-basis}) at each coarse node to construct GMsFE space and denote the $M$-th eigenvalue by $\lambda_M$,
then under certain assumptions (see \cite{ye11}),  the solution error $\mathcal{E}_{ms}$ for the subsurface model by GMsFEM  is
\[
\mathcal{E}_{ms}=O\big(\sqrt{H\over \lambda_M}+H\big),
\]
where $H$ is the size of coarse cells.
\end{rem}

%%%%%%%%%%%%%%%%%%%%%%%%%%%%%%%%%%%%%%%%%%%%%%%%%%%%%%%%%%%

\section{Numerical examples}
In this section, we use GMsFEM and LS-SCM to build a reduced computational model  for the equation (\ref{flow-eq}),  and recover the model's inputs using Bayesian framework.
In Subsection \ref{rec-init}, we recover initial condition. Subsection \ref{rec-loc} is devoted to the inversion of source location.
In Subsection \ref{rec-loc-flux}, we combine the techniques presented in Subsection \ref{rec-init} and Subsection \ref{rec-loc}, and recover source location and flux simultaneously.
In all our numerical examples, we consider the model equation (\ref{flow-eq}) in a high-contrast permeability field, whose profile is depicted in   Figure \ref{perm}.
 For the numerical examples,  we consider a dimensionless  square domain $\Omega:=[0,1]\times[0,1]$ for space, and $(0, T]=(0, 0.1]$ for time.
We will compare the results by using full-order model and the reduced order model, and analyze the approximation for different estimated parameters.

\begin{figure}[htbp]
  \centering
  \includegraphics[width=0.6\textwidth]{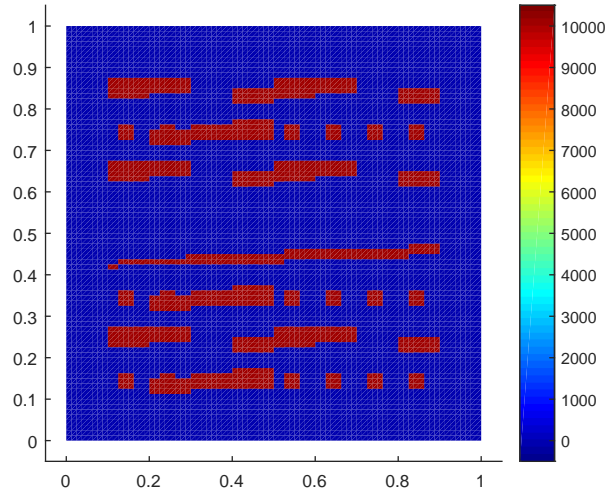}
  \caption{The spatial distribution of the high contrast coefficient $k(x)$}
  \label{perm}
\end{figure}

\subsection{Recover initial condition}
\label{rec-init}

  In this subsection, we want to recover the initial condition based on some measurements.
  We consider the model with the  following boundary condition and source term
  \begin{equation}
  \label{ex1}
    \left\{
  \begin{aligned}
  \frac{\partial u}{\partial t}&=\text{div}\bigg(k(x)\nabla u(x,t)\bigg)+f,\ x\in\Omega,t\in(0,T]\\
  k\frac{\partial u}{\partial n}&=0,\ \text{on}\ \partial\Omega, t\in(0,T]\\
  u(x,0)&=u_0(x),\ \ x\in\Omega,
  \end{aligned}
     \right.
  \end{equation}
where the source term is given by
  \[
  f(x)=10\exp\bigg(-\frac{(x_1-0.2)^2+(x_2-0.4)^2}{2\times0.2^2}\bigg).
  \]

The goal of the example is to reconstruct the initial condition $u_0(x)$.
For simulation, we need to represent the function $u_0(x)$   in a finite dimensional space.  To this end, we project the initial function onto a subspace spanned by $m$ finite element basis functions $\{\xi_i\}_{i=1}^m$. Thus  the initial condition can be represented as
  \begin{equation}
  \label{initial}
    u_0(x)=\sum_{i=1}^{m} z_i\xi_i(x).
  \end{equation}
 In the paper,  we take $\xi_i(x)$ to be  the bilinear finite element basis functions.  The parameter $z\in\mathbb{R}^m$, where the dimension  depends on the discretization. The finer the discretization is, the better the approximation is. But this in turn leads to a larger dimension of $z$. We note that the high dimension of estimated parameters would increase the mixed time of Markov chains. To overcome the difficulty, we will use Karhunen-Loeve expansion (KLE) technique to reduce the dimension.

As the initial condition is  spatially varying and unknown  to us, we can treat it as a random field $u_0(x,\omega)$.  We use KLE  and the random field $u_0(x,\omega)$ can be represented as
  \begin{equation}
  \label{initial_kle}
    u_0(x,\omega)=\mathbb{E}[u_0(x,\omega)]+\sum_{i=1}^{\infty} \sqrt{\zeta_i}\eta_i(\omega)\varphi_i(x),
  \end{equation}
 where $\varphi_i$ are the orthogonal eigenfunctions and $\zeta_i$ are the corresponding eigenvalues of the eigenvalue problem,
  \[
  \int_\Omega C(x,x')\varphi_i(x')dx'=\zeta_i\varphi_i(x),\quad \quad i=1,2,\cdots.
  \]
Here  $C$ is the covariance function defined by
  \[
  C(x,x')=\mathbb{E}[u_0(x,\omega)u_0(x',\omega)].
  \]
We note that $\mathbb{E}[\eta_i]=0$, $\mathbb{E}[\eta_i\eta_j]=\delta_{ij}$, and bilinear finite element basis functions  are used to discretize the eigenvalue problem described above. We truncate the KLE $(\ref{initial_kle})$ to a finite number of terms and keep only the leading-order terms to capture most of the energy of the stochastic process.
We truncate the first $m_0-$ terms for the approximation
   \[
    u_0(x,\omega)\approx  \mathbb{E}[u_0(x,\omega)]+ \sum_{i=1}^{m_0} \sqrt{\zeta_i}\eta_i(\omega)\varphi_i(x).
  \]
The energy ratio of the approximation is defined by
  \[
  e(m_0)=\frac{\sum_{i=1}^{m_0}\zeta_i}{\sum_{i=1}^{\infty}\zeta_i},
  \]
and we set $e(m_0)=0.95$ in the examples.  Then the relationship between $z$ and $\eta$ can be expressed by
  \begin{equation}
  \label{kle_relation}
  \theta=\mathbf{B}\eta,
  \end{equation}
where $\eta\in\mathbb{R}^{m_0}$ and $\mathbf{B}\in\mathbb{R}^{m\times m_0}$, which is defined by
  \[
  \mathbf{B}=\big[\sqrt{\zeta_1}\varphi_1, \sqrt{\zeta_2}\varphi_2, \cdots, \sqrt{\zeta_{m_0}}\varphi_{m_0}\big].
  \]

As the solution of $(\ref{ex1})$ depends linearly on the initial function, we have the following approximation
  \begin{equation}
  \label{H-I}
  u(x)=Hz+I,
  \end{equation}
where $H$ is the sensitivity matrix \cite{jw04} defined by
  \[
  H=\big[u(\xi_1(x))\quad u(\xi_2(x))\quad \cdots\quad u(\xi_m(x))\big].
  \]
Here $u(\xi_i(x))\in\mathbb{R}^{n_d}$ denotes the solution at measured sensor network with initial condition $\xi_i(x)$, zero source term and homogeneous Neumann boundary condition. The $I$ in equation (\ref{H-I}) represents the solution when initial condition and boundary condition is $0$ but the source term is $f$. The sensitivity matrix is required to be full column rank here. Let $z\in\mathbb{R}^m$ be the coefficient in $(\ref{initial})$. Then the relation between the unknown vector $z$ and observation $d$ is:
  \[
   d=Hz+I+e,
  \]
where $e$ is the Gaussian noise with standard deviation $\sigma$. Thus  the likelihood function is given by
  \[
  L(z)=(2\pi\sigma^2)^{-\frac{n_d}{2}}\exp\bigg(-\frac{\|d-Hz-I\|_2^2}{2\sigma^2}\bigg).
  \]
We use  a Gibbs sampler for the case when Markov Random Field (MRF) \cite{jw06} is selected as the prior density in this example. The MRF takes the form,
  \[
   \pi(z)\propto \gamma^{m/2}\exp(-\frac{1}{2}\gamma z^TWz),
  \]
where the entries of the $m\times m$ matrix $W$ is specified as following: $W_{ij}=n_i$ if $i=j$, $W_{ij}=-1$ if $i$ and $j$ are adjacent, and as $0$ otherwise. Here the $n_i$ is the number of neighbors adjacent to site $i$. In general, the neighbors to a particular unknown at a given location of a finite lattice refer to unknowns at adjacent points on the same lattice. $W$ determines the dependence between components of $z$, and various dependence relations among variables can be characterized  by changing the form of $W$. The $\gamma$ controls the strength of spatial dependence and regularization to the inverse problem, which should be tuned relying on one's experience. We treat it as a hyperparameter and choose Gamma distribution as the its hyperprior density. It can be used as conjugate prior distribution \cite{pc07} here, i.e.,
  \[
  \pi(\gamma)=\frac{\beta_1^{\alpha_1}}{\Gamma(\alpha_1)}\gamma^{\alpha_1-1}\exp(-\beta_1\gamma),\quad\gamma>0,\quad\alpha_1>0,\beta_1>0,
  \]
where $\alpha_1$ is the shape and $\beta_1$ is the rate. Then the joint posterior density is
  \[
  \pi(z,\gamma|d)\propto\frac{\beta_1^{\alpha_1}}{\Gamma(\alpha_1)}\gamma^{\alpha_1+\frac{m}{2}-1}\exp\{-(\beta_1+\frac{1}{2}z^TWz)\gamma\}
                        \exp\bigg(-\frac{\|d-I-Hz\|^2}{2\sigma^2}\bigg),
  \]
and the conditional posterior distributions can be derived as
  \begin{equation}
  \label{initial_z}
  \pi(z|d,\gamma) \propto \exp\bigg(-\frac{\|d-I-Hz\|^2}{2\sigma^2}\bigg)\exp(-\frac{\gamma}{2}z^TWz),
  \end{equation}
  \begin{equation}
  \label{initial_la}
   \pi(\gamma|d,z) \sim \Gamma(\alpha_1+\frac{m}{2}, \beta_1+\frac{1}{2}z^TWz).
  \end{equation}
For the convenience of notation, we denote $\mathbf{H}:=H\mathbf{B}$, and $\mathbf{W}:=\mathbf{B}^TW\mathbf{B}$, substitute $(\ref{kle_relation})$ into $(\ref{initial_z})$ and $(\ref{initial_la})$, we obtain the conditional posterior distributions of $\eta$ and $\gamma$, respectively,
  \begin{eqnarray*}
  \pi(\eta|d,\gamma) &\propto& \exp\bigg(-\frac{\|d-I-\mathbf{H}\eta\|^2}{2\sigma^2}\bigg)\exp(-\frac{\gamma}{2}\eta^T\mathbf{W}\eta),\\
  \pi(\gamma|d,\eta) &\sim& \Gamma(\alpha_1+\frac{m_0}{2}, \beta_1+\frac{1}{2}\eta^T\mathbf{W}\eta).
  \end{eqnarray*}
We note that the posterior distribution of $\gamma$ is  easy to update during the Gibbs sampling. In addition, we set parameters $\alpha_1$ and $\beta_1$ in the hyperprior density small so that hyperprior density can nearly be a uniform distribution among the interval $(0, +\infty)$, e.g., $\alpha_1=\beta_1=0.001$. This is the so called noninformative prior.
When the noise level $\sigma$ is known, the posterior distribution of $\eta$ follows a multivariate Gaussian distribution. Moreover, the full conditional distribution of each component $\eta_i$ is in standard form and can be derived as follows \cite{jw04},
  \[
   \pi(\eta_i|\eta_{-i})\sim N(\mu_i,\sigma_i^2),
  \]
where $\eta_{-i}:=(\eta_1, \cdots, \eta_{i-1}, \eta_{i+1}, \cdots, \eta_{m_0})$ and
  \begin{eqnarray*}
  \sigma_i &=& (\sum_{k=1}^n \frac{\mathbf{H}_{ki}^2}{\sigma^2}+\gamma \mathbf{W}_{ii})^{-\frac{1}{2}},\\
  \varrho &=& \sum_{j\neq i} \mathbf{W}_{ji}\eta_j+\sum_{k\neq i} \mathbf{W}_{ik}\eta_k,\\
  \rho_k &=& d_k-(I)_k-\sum_{j\neq k}\mathbf{H}_{kj}\eta_j,\\
  \mu_i &=& \frac{2\sum_{k=1}^n \frac{\rho_k \mathbf{H}_{ki}}{\sigma^2}-\gamma\varrho}{2(\sum_{k=1}^n \frac{\mathbf{H}_{ki}^2}{\sigma^2}+\gamma \mathbf{W}_{ii})}.
  \end{eqnarray*}
Hence, we treat $[\gamma;\eta]$ as big block and update the component of $\eta$ as small block during the sampling, and we get samplers from $\eta-$space and transform them back to $z-$space. As we have discussed, the Gibbs sampler method can provide us much efficiency in simulation, once the sensitivity matrix $H$ and rest term $I$ calculated, we do not need to solve the forward model any more during the simulation, and the computation here is the calculation of the sensitivity matrix $H$.

The forward model is  solved on a uniform $80\times80$ fine grid. If we resolve all scales and solve the forward model in the fine grid, a linear system of equations with $6561$ unknowns would be required to be solved at each time layer during the iteration, and we  have the iterations $m+1$ time levels. In order to reduced the number of unknowns, we use GMsFEM to compute the model, i.e., we use GMsFEM to obtain a sensitivity matrix $H^M$.  The matrix  $\mathbf{H}$ is computed based on  $H^M$ and $\mathbf{B}$ during the sampling of $\eta$. We set $8\times8$ coarse grid for GMsFEM, and select $8$ multiscale basis functions $M$ on each coarse neighborhood. Then the dimension of unknowns solving the PDE at each time layer decreases to $648$.

\begin{figure}[h]
    \centering
    \includegraphics[width=0.6\textwidth]{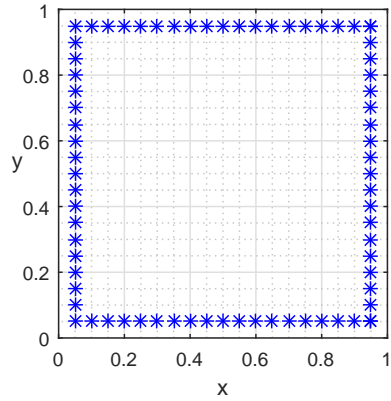}\\
    \caption{The distribution of measurement locations (marked by $*$) in the domain $\Omega$}
\label{mea-loc}
\end{figure}

The observation  data are taken from time $[0.01:0.01:0.1]$, and we measure the finite element solution at points shown as Figure \ref{mea-loc}, with the noise $\sigma=0.01$ in this example. We assume the covariance function has the form
  \[
  C(x_1,x_2;x'_1,x'_2)=\varsigma^2\exp\bigg(-\frac{|x_1-x'_1|^2}{2l_1^2}-\frac{|x_2-x'_2|^2}{2l_2^2}\bigg)
  \]
with $l_1=l_2=0.2$ and $\varsigma^2=2$. We set the true initial function as follows,
\[
u_0(x)=\cos(\pi x_1)\cos(\pi x_2)+1.5.
\]
$u_0$ is represented in a $11\times 11$ grid. In truncated KLE for $u_0$, the dimension of $\eta$ is only $75$, i.e,  $m_0=75$, which is much smaller than the original dimension  $m=11^2$.
The relative $L_2$ error between true initial condition $u_0$ and estimated initial condition $\hat{u}_0$  is defined by
  \[
  r=\frac{\|u_0-\hat{u_0}\|_{L_2}}{\|u_0\|_{L_2}},
  \]
where $\hat{u_0}$ refers to
  \[
   \hat u_0(x)=\sum_{i=1}^{m} \hat{z_i}\xi_i(x)
  \]
and $\hat{z_i}$ is the estimator of $z_i$.
The profiles of true initial condition and estimated initial condition are shown in Figure \ref{comp-init}, from which we can see an accurate estimate for the initial condition.
 \begin{figure}[htbp]
    \centering
    \includegraphics[width=0.7\textwidth]{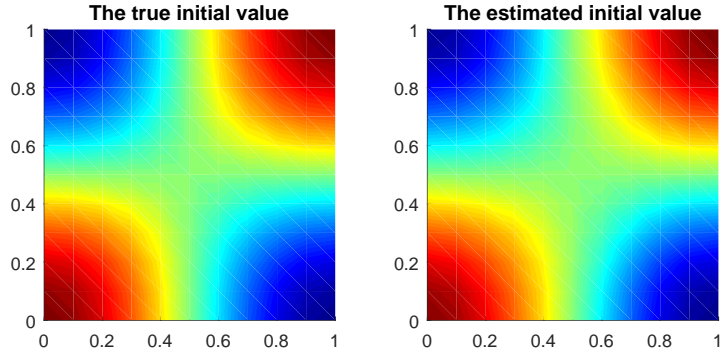}
    \caption{The true initial value (left) vs. the estimated initial value (right),  the relative $L_2$ error is 1.33\% with 1\% noise in the data}
    \label{comp-init}
\end{figure}

 Figure \ref{post-gamma} shows the posterior marginal density of the hyperparameter $\gamma$, where the solid line is obtained by using the full-order model, dashed line is obtained using the surrogate model constructed by GMsFEM, the green one is obtained using $8$ multiscale basis functions per coarse node and the red one using $8$ multiscale basis functions per coarse node. We see in the  figure  that the  distribution  from surrogate model approximates the distribution of full-order model  better as the number of multiscale basis functions increases.

\begin{figure}[h]
    \centering
    \includegraphics[width=0.6\textwidth]{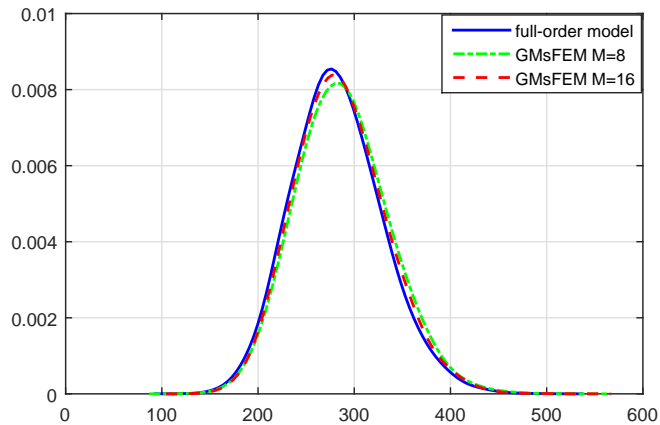}
    \caption{The posterior marginal densities of $\gamma$ with 1\% noise in the data}
   \label{post-gamma}
\end{figure}

We also solve the inverse problem using GMsFEM and reconstruct the initial condition with different noise level.
  Let noise to signal ratio be defined by $\sigma/u_{max}$, where $u_{max}$ is the maximum of solution $u$.
  We list the results in Table \ref{tab2}. From the table,
   we  see that when the noise level is $\sigma=0.01$, the relative error is about $1.33\%$, which has small difference from the case $\sigma=0.001$ but  much difference from the case $\sigma=0.1$. The case noise level $\sigma=0.5$ leads to a big error. This shows that the error of inversion  increases as the noise level increases. The more the measurement error, the worse the estimation is.  As the measurement error is within some appropriate range, we can reconstruct the initial value well.
\begin{table}[hb]
  \centering
  \caption{Results comparison with different measurement noise}\label{tab2}
  \begin{tabular}{c|c|c}
  \hline
   Noise level & Noise to signal ratio & Relative error \rule{0pt}{1cm} \\
   \hline
   0.001 & 0.04\% &1.26\%  \rule{0pt}{0.5cm}\\
  0.01 & 0.4\% & 1.33\%  \rule{0pt}{0.5cm}\\
  0.1 & 4\% & 12.44\% \rule{0pt}{0.5cm}\\
  0.5& 20\% &22.95\% \rule{0pt}{0.5cm}\\
  \hline
\end{tabular}
\end{table}

\begin{figure}[htbp]
    \centering
    \includegraphics[width=0.8\textwidth,height=8in]{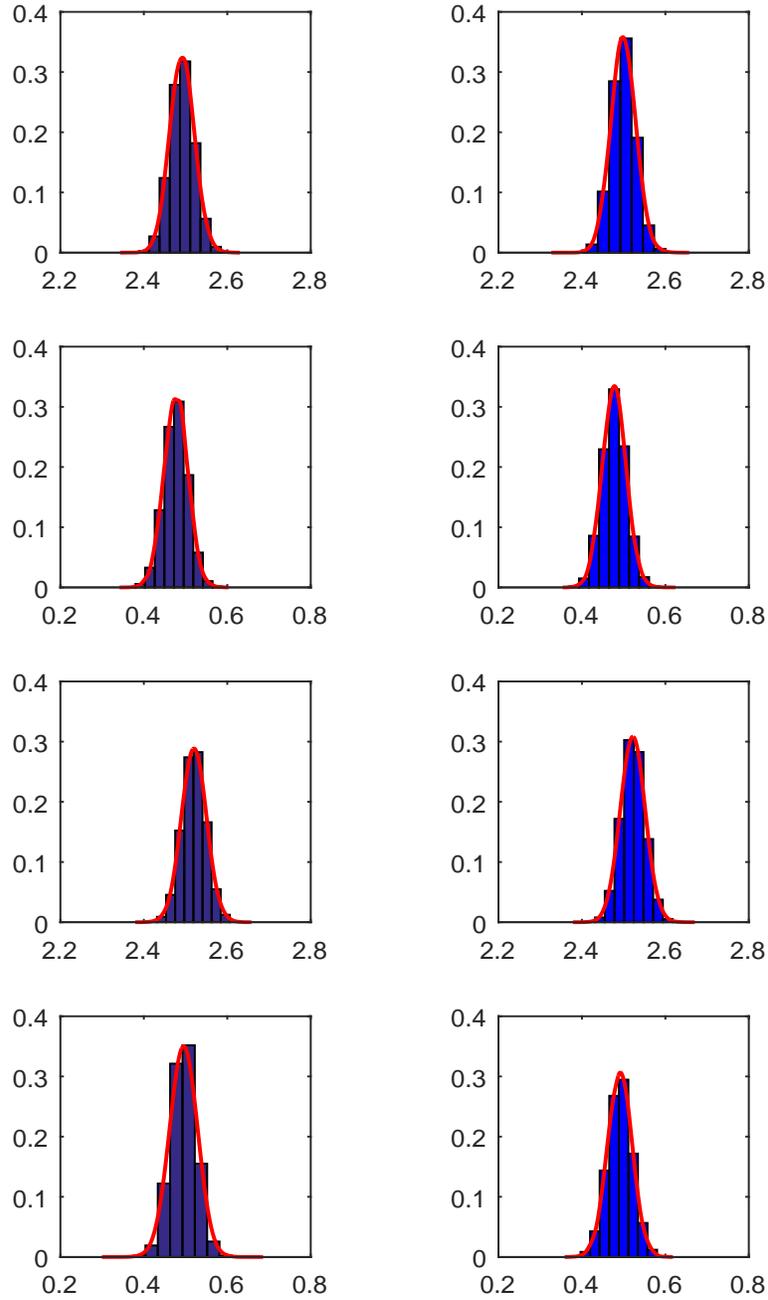}\\
    \caption{The posterior marginal distribution of unknowns at points $u_0(0, 0)$ (first row), $u_0(1, 0)$ (second row), $u_0(1, 1)$ (third row)and $u_0(0, 1)$ (fourth row), and the left column characterises series from the full model while the right column characterises series from the reduced model, $\sigma=0.01$ and $M=8$}
    \label{8-dist}
\end{figure}

Figure \ref{8-dist} plots  the marginal distributions at the $4$ corners using full-order model and reduced order model. In the figure,  the left column refers to samples of the chain constructed from the full-order model and the right column refers to samples of the chain constructed from the reduced-order model. In both runs of the Gibbs sampler, $30000$ samples of $\eta$ are recorded and the last 20000 are used to compute the distributions. We transform them back to $z$ in the plot. It can be seen that the posterior mean estimates have a good agreement using the full-order model and reduced-order model. Moreover, we use Kullback-Leibler divergence to quantify the difference between the approximated joint posterior density and the reference joint posterior density. Though $\gamma$ is part of the inference, $z$ is the main interest for us, we focus on the posterior density of $z$. Denote $L^M$ the approximate likelihood function and
  \[
   L^M(z)=(2\pi\sigma^2)^{-\frac{n_d}{2}}\exp\bigg(-\frac{\|d-H^Mz-I\|_2^2}{2\sigma^2}\bigg).
  \]
  We integrate the joint posterior density with respect to $\gamma$ and the marginal posterior densities of $z$ for full-order model and reduced-order model are given, respectively,  by
  \begin{eqnarray*}
  \pi^d(z) &=& \frac{\int L(z)\pi(z|\gamma)\pi(\gamma)d\gamma}{\int L(z)\pi(z|\gamma)\pi(\gamma)d\gamma dz}\\
  &=& \frac{L(z)S(z)}{\int L(z)S(z)dz},\\
  \pi^{d,M}(z) &=& \frac{\int L^M(z)\pi(z|\gamma)\pi(\gamma)d\gamma}{\int L^M(z)\pi(z|\gamma)\pi(\gamma)d\gamma dz} \\
  &=& \frac{L^M(z)S(z)}{\int L^M(z)S(z)dz},\\
  \end{eqnarray*}
where $S(z)$ is given by
  \[
  S(z)=\big(\frac{z^TWz}{2}+\beta_1\big)^{-(\frac{m}{2}+\alpha_1)}.
  \]
The normalized term in the exact posterior density can be rewritten as
  \begin{eqnarray*}
  \int L(z)S(z)dz &=& \int \frac{L(z)S(z)}{\pi^{d,M}(z)}\pi^{d,M}(z) dz\\
  &=&  \int \frac{L(z)S(z)\int L^M(z)S(z)dz }{L^M(z)S(z)}\pi^{d,M}(z)dz,
  \end{eqnarray*}
 rearrange the equation, we have
  \[
  \frac{\int L(z)S(z)dz}{\int L^M(z)S(z)dz}=\int \frac{L(z)}{L^M(z)}\pi^{d,M}(z)dz.
  \]
Hence the Kullback-Leibler divergence can be rewritten as
   \[
   D_{KL}(\pi^{d,M}||\pi^d)=\mathbb{E}_{\pi^{d,M}}[\log \frac{L^M}{L}]+\log \mathbb{E}_{ \pi^{d,M}}[\frac{L}{L^M}],
   \]
where $z^{(j)}$ are independent samplers from $\pi^{d,M}$. When the GMsFEM is used to solve the forward model on a fixed coarse grid, the approximation
accuracy for the forward model depends on the number of multiscale basis functions we select on each coarse neighborhood \cite{ye11}. Here we discuss the effect of number of multiscale basis functions on the KL divergence $D_{KL}$. In Figure \ref{KL-5.1}, $D_{KL}$ is plotted against  number of multiscale basis functions per coarse node. Here the noise level is $0.01$.
By the figure, we find that the posterior density of Kullback-Leibler divergence $D_{KL}$ decreases as we increase the number of multiscale basis functions per node. This implies that the posterior distribution by surrogate model  approximates the reference posterior distribution better and better as we enrich multiscale basis functions.
\begin{figure}[h]
  \centering
  \includegraphics[width=0.6\textwidth]{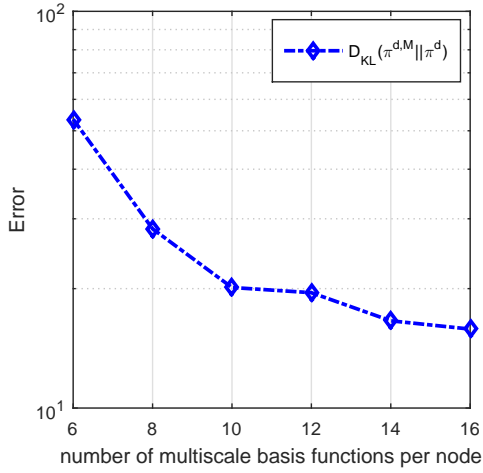}
 \caption{Kullback-Leibler divergence $D_{KL}(\pi^{d,M}||\pi^d)$ between the approximate posterior and the reference posterior}
 \label{KL-5.1}
\end{figure}

%%%%%%%%%%%%%%%%%%%%%%%%%%%%%%%%%%%%%%%%%%%%%%%%%%%%%%%%%%%%%%

\subsection{Recover source location}
\label{rec-loc}

 In this subsection, we focus on the inversion of source location.
 We still consider the equation (\ref{ex1}) with the initial condition  $u_0(x)=0$, and the source term is
  \[
  f(x)= \frac{5}{2\pi\times0.1^2}\exp\bigg(-\frac{(x_1-z_1)^2+(x_2-z_2)^2}{2\times0.1^2}\bigg),
  \]
where $(z_1,z_2)$ denotes the location of the one point source, with strength 5 and width 0.1. We need to identify the source location $(z_1, z_2)$.

The location of the source enters the problem non-linearly, which implies the explicit expression of the posterior distribution is unavailable, large scale PDE problems are required to be solved repeatedly for proposal samplers, which brings up the main computation burden. We use the truncated gPC to approximate the forward model at some observation sensors, and then replace the forward model with the established surrogate model to obtain samplers via MH algorithm. As referred in section \ref{SCM}, when LS-SCM applied, large numbers of deterministic forward models are required to be solved at the off-line stage, we use GMsFEM to solve the corresponding problems to calculate the sample vector $b$ in (\ref{ls}).

The parameter $z=(z_1, z_2)$ is  unknown, and we assume the uniform distribution as its  prior density, i.e., $Z_i\sim U(0,1), i=1,2$. For any given values of $z$, we solve the PDE on a uniform $40\times40$ fine grid
using GMsFEM with time step $\Delta t=0.004$. Observation data are generated by adding independent random noise $N(0, \sigma^2)$ to the solution at a uniform $6\times6$ sensor network. At each sensor location, measurements are taken at time $t=0.04, 0.08$, which corresponds to a total of $72$ measurements. To avoid ``inverse crime", we generate the data by solving the forward model at a much higher resolution than that used in the inversion, i.e., using finite element method at the fine grid and a correspondingly finer time step $\Delta t=0.002$.

The ground truth parameter values is set as $z=(0.25, 0.75)$ in the example. When constructing the gPC approximation, we set the coarse grid size as $N_v=5\times5$ to solve the forward model. Samples $\{z^{(i)}\}_{i=1}^{Q}$ are selected randomly from the prior distribution to construct the marginal matrix $V$ and sample vector $b$, the number of samples is set according to $(\ref{ls_rule})$.

Figure \ref{pos_M} and \ref{pos_N} show the contours of the likelihoods with $\sigma=0.05$, where the solid and dashed lines denote the reference likelihood $L(z)$ and the surrogate likelihood $\tilde L_N^M(z)$, respectively. The accuracy of the surrogate likelihood constructed by  GMsFEM and LS-SMC depends on the number of multiscale basis functions per node and order of the gPC expansion. When the order of the gPC expansion is fixed at $N=10$, Figure \ref{pos_M} shows the difference between the approximate and reference posterior distribution with $6$, $14$ and $22$ multiscale basis functions per node, respectively. When the number of multiscale basis functions is fixed at $M=14$ per node, Figure \ref{pos_N} illustrates the difference between the approximate and reference distribution with different gPC order $4$, $6$, and $8$. The better agreement between $\tilde \pi_N^{d,M}(z)$ and $\pi^d(z)$ is observed with increasing the multiscale basis functions number and gPC order.
\begin{figure}
  \centering
  \includegraphics[width=\textwidth]{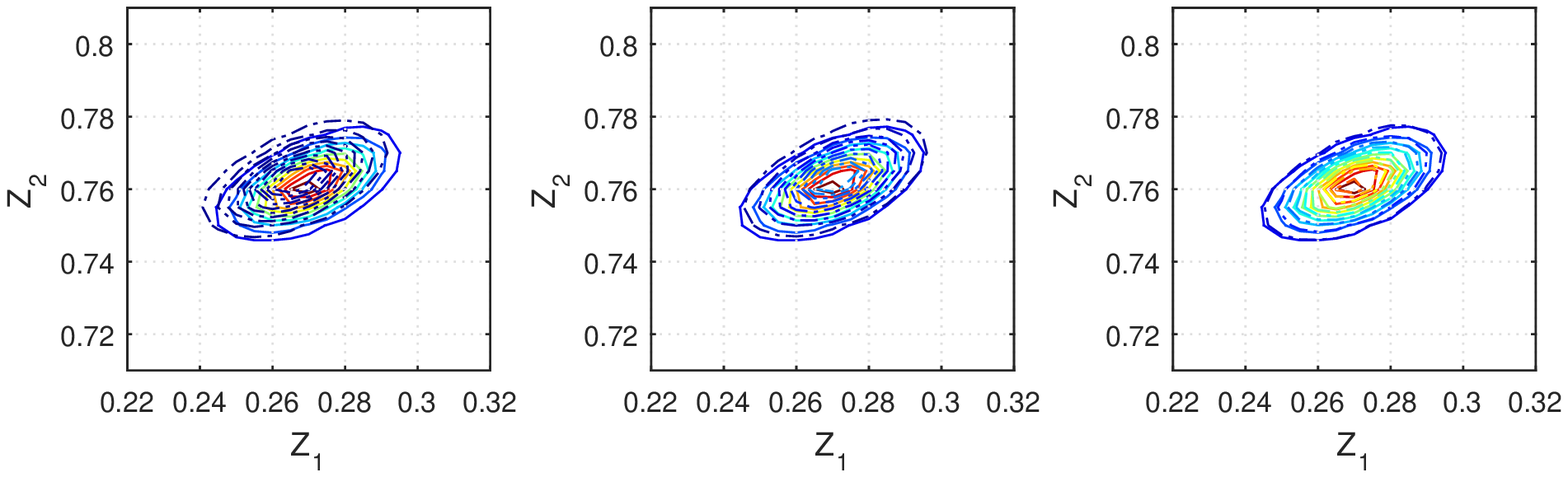}
  \caption{Contours of posterior density of source location. Solid lines are obtained via full forward model; dash lines are obtained via the reduced order model, $N=10$ is fixed, M=6 (left) M=14 (middle),  M=22 (right).}\label{pos_M}
\end{figure}
\begin{figure}
  \centering
  \includegraphics[width=\textwidth]{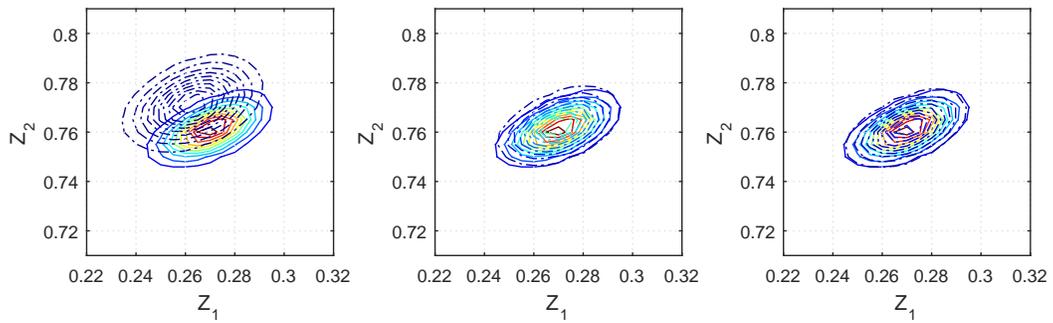}
  \caption{Contours of posterior density of source location. Solid lines are obtained via full forward model; dash lines are obtained via the reduced order model, $M=14$ is fixed, N=4 (left) N=6 (middle), N=8 (right).}\label{pos_N}
\end{figure}

The scale of Random walk method we used in the simulation is $0.005$ and the length of each Markov chain is $30000$, only the last $20000$ realizations are used to compute the relevant statistical quantities. We set $M=14$, $N=8$, the marginal posterior distribution of $z_1$ and $z_2$ are shown in Figure \ref{marge}, where the red solid lines are obtained with $\sigma=0.01$ and blue dashed lines are obtained with $\sigma=0.05$. Thanks to the informed likelihood or measurement data, the posterior support of each parameter is narrower than their priors. As what we expect, the support of the posterior distribution derived by data with noise $0.01$ is narrower than the one derived by data with noise $0.05$.
\begin{figure}
  \centering
  \includegraphics[width=0.7\textwidth]{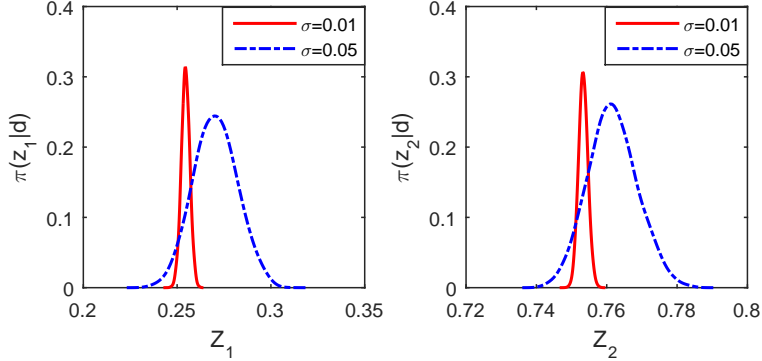}
  \caption{The posterior marginal densities of  $Z_1$ (left) and $Z_2$ (right) with $\sigma=0.01$ and $\sigma=0.05$ noise in the data}\label{marge}
\end{figure}

 The accuracy of the surrogate model depends both on $M$ and $N$. We discuss  the KL divergence $D_{KL}$ between the approximated posterior measure and the reference posterior measure. First we consider the effect of the number of multiscale basis functions per node $M$ on the performance. In Figure \ref{Kld_M}, the approximation of the surrogate model constructed by combining GMsFEM and LS-SCM is plotted against increasing numbers of selected multiscale basis functions when the order of the polynomial is fixed at $N=10$. From this figure,  the model approximation $\sum_{i=1}^{n_d}\|G_i-\tilde G_{N,i}^M\|^2_{L_{\pi_z}^2}$ becomes better as GMsFE basis functions enrich. This is consistent  with the result shown in Theorem \ref{G-GM}. The top curve of the figure shows that the posterior density of KL divergence $D_{KL}$ decreases as we increase the selected multiscale basis functions.  There exists a slight fluctuation  from $M=6$ to $M=8$. This may be caused by the different samples used in constructing the surrogate model at the off-line stage. Convergence of the posterior with respect to gPC order is analyzed in Figure \ref{Kld_N}, where the number of multiscale basis functions per node is fixed at $M=14$. The error $\sum_{i=1}^{n_d}\|G_i-\tilde G_{N,i}^M\|^2_{L_{\pi_z}^2}$ and KL divergence decreases as the gPC order $N$ increases.
\begin{figure}[htbp]
  \centering
  \includegraphics[width=0.6\textwidth]{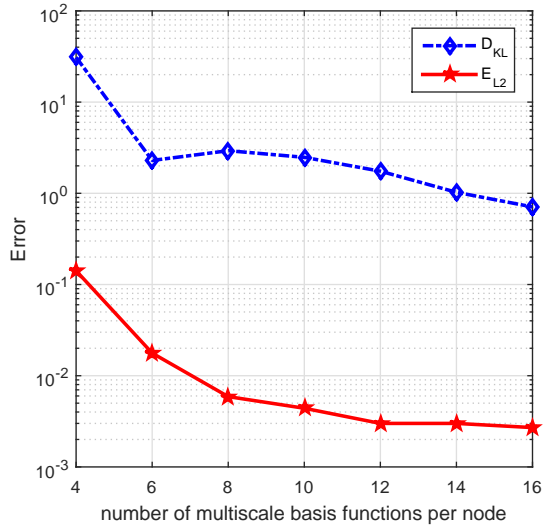}
  \caption{Approximation of the forward model and the posterior density with respect to the number of  multiscale basis functions per node. Dashed line show Kullback-Leibler divergence $D_{KL}(\tilde \pi_N^{d,M}||\pi^d)$ (denoted by $D_{KL}$); solid line show $L_{\pi_z}^2$ error  $\sum_{i=1}^{n_d}\|G_i-\tilde G_{N,i}^M\|^2_{L_{\pi_z}^2}$ (denoted by $E_{L_2}$).}\label{Kld_M}
\end{figure}
\begin{figure}[htbp]
  \centering
  \includegraphics[width=0.6\textwidth]{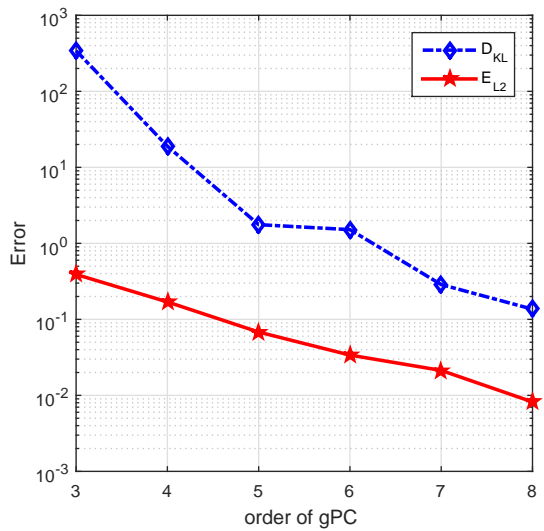}
  \caption{Approximation of the forward model and the posterior density with respect to  the gPC order $N$}\label{Kld_N}
\end{figure}

%%%%%%%%%%%%%%%%%%%%%%%%%%%%%%%%%%%%%%%%%%%%%%%%%%%%%%%%%%%%

\subsection{Recover source location and flux simultaneously}
\label{rec-loc-flux}

 In this subsection, we  reconstruct the boundary flux  and identify  the source location simultaneously with some measured data. We consider the model with the following boundary condition and source term
  \begin{equation}
  \label{ex3}
    \left\{
    \begin{aligned}%\label{}
    \frac{\partial u}{\partial t}&=\text{div}\bigg(k(x)\nabla u(x,t)\bigg)+f(x), \quad x\in\Omega,t\in(0,T]\\
      k\frac{\partial u}{\partial n} &=p(x_2,t),\ \text{on}\ \Gamma_{N_1},t\in(0,T]\\
      k\frac{\partial u}{\partial n} &=0,\ \text{on}\ \Gamma_{N_2},t\in(0,T]\\
      u(x,t)&=0,\ \text{on}\ \Gamma_D,t\in(0,T]\\
      u(x,0)&=0,\ \ x\in\Omega,
    \end{aligned}
    \right.
  \end{equation}
where the boundaries are $\Gamma_{N_1}=\{(0,x_2)\in\Omega\}$, $\Gamma_{N_2}=\{(x_1,0)\in\Omega\}$, and $\Gamma_D=\{(1,x_2)\in\Omega,(x_1,1)\in\Omega\}$.

We assume the true flux and source term are
  \begin{eqnarray*}
    p(x_2,t)&=&\sin(\pi x_2)\sin(10\pi t)+5, \\
    f(x) &=&  \frac{8}{2\pi \times0.1^2}\exp\bigg(-\frac{(x_1-z_1)^2+(x_2-z_2)^2}{2\times0.1^2}\bigg),
  \end{eqnarray*}
where $z=(z_1,z_2)\in\mathbb{R}^2$ is the unknown source location. We want to estimate the flux and source location.   Note that the flux depends on time $t$ and enters the system linearly, the location of source term enters the system nonlinearly.

The unknown flux function $p(x_2,t)$ can be  discretized  in space and time.  Let $(x_2,t):=s$. The unknown flux can be treated as a random field $p(s,\omega)$. We denote $\eta\in\mathbb{R}^{n_0}$ as the KLE coefficient vector of the discretized flux function. The uncertainty of the system comes from $\eta$ and $z$. In the Bayesian setting, both $\eta$ and $z$ are random variables. It is natural to suppose that they are independent of each other. If a truncated gPC expansion is applied to approximate the system $(\ref{ex3})$ directly, the high dimension of the unknowns will lead great challenge for  solving the inverse problem. We consider another expansion of the forward model that separates $\eta$ and $z$.  Due to the linearity of the parabolic PDE, the forward model has the decomposition,
  \[
  G(\eta,z)=f_1(\eta)+f_2(z),
  \]
where $f_1(\eta)$ denotes the solution of the system $(\ref{ex3})$ with zero source term, and $f_2(z)$  the solution of the system $(\ref{ex3})$ with $p(x_2,t)=0$.  Our goal is to find the appropriate estimation of $\eta$ and $z$ given the measured data $d$. Due to the ill-posedness of the problem, some prior information is needed. Following the examples in Subsection \ref{rec-init} and Subsection \ref{rec-loc}, we use MRF as the prior for the flux parameter, and uniform distribution for the prior of the location. Then the posterior density can be derived as
  \[
  \pi(\eta,z|d)\propto \gamma^{\frac{n_0}{2}}\exp(-\frac{\gamma\eta^T\mathbf{W}\eta}{2})\exp\{-\frac{1}{2\sigma^2}\|d-f_2(z)-f_1(\eta)\|^2\},
  \]
where $\gamma$ is a hyperparameter.  We use  the usual conjugate gamma prior $\Gamma(\alpha_2, \beta_2)$.  By using  the symbol defined  in Subsection \ref{rec-init} $\mathbf{W}$, the final conditional posterior density has the form
  \begin{equation}
  \label{flux_eta}
  \pi(\eta|d,z,\gamma) \propto \exp\{-\frac{1}{2\sigma^2}\|d-f_2(\chi)-f_1(\eta)\|^2\}\exp\{-\frac{\gamma}{2}\eta^T\mathbf{W}\eta\},
  \end{equation}
  \begin{equation}
  \label{source}
  \pi(z|d,\eta,\gamma) \propto \exp\{-\frac{1}{2\sigma^2}\|d-f_2(z)-f_1(\eta)\|^2\},
  \end{equation}
  \begin{equation}
  \label{lamda}
  \pi(\gamma|d,\eta,z) \sim \Gamma(\alpha_2+\frac{n_0}{2}, \beta_2+\frac{1}{2}\eta^T\mathbf{W}\eta).
  \end{equation}
As we have noticed that the flux enters the model linearly, a sensitivity matrix $H_2$ can be obtained similarly  as in Subsection \ref{rec-init}, i.e.,
\[
f_1(\eta)=\mathbf{H}_2\eta, \quad \mathbf{H}_2:=H_2\mathbf{B}.
 \]
 In order to estimate the source location efficiently, we use GMsFEM and LS-SCM to construct a reduced order model for  $(\ref{ex3})$ with zero boundary condition  and zero initial condition. Then the reduced order model for  $(\ref{ex3})$ can be expressed by
  \[
  \tilde G_N^M(\eta, z)=\mathbf{H}^M_2\eta+\tilde G_N^M(z),
  \]
where $\mathbf{H}^M_2:=H^M_2\mathbf{B}$ and $H^M_2$ is the sensitivity matrix computed  by GMsFEM,  $\tilde G_N^M(z)$ is the surrogate model of system only dependent on $z$, which is constructed by combing GMsFEM with LS-SCM. Inspired by the numerical experiments in Subsection \ref{rec-init} and Subsection \ref{rec-loc}, we propose to use Gibbs  method to sample the flux, and random walk method to the source location $z$.  The outline of the computation  for the example is described in  Table \ref{tab3}.

\begin{table}[htbp]
  \centering
  \caption{The outline of the computation for the example in Subsection \ref{rec-loc-flux}}\label{tab3}
  \begin{tabular}{l}
     \hline
  \bf{Off-line phase/Construction of the reduced order model}:\rule{0pt}{0.8cm}\\
  $\cdot$ Calculate the GMsFEM matrix $R$;\rule{0pt}{0.5cm}\\
  $\cdot$ Use GMsFEM to obtain the sensitivity matrix $H^M_2$;\rule{0pt}{0.5cm}\\
  $\cdot$ Use finite element basis functions to discretize the eigenvalue problem involved in KLE and \\
    \quad obtain matrix $\mathbf{B}$;\rule{0pt}{0.5cm}\\
  $\cdot$ Combine GMsFEM with LS-SCM to obtain the approximation $\tilde G_N^M(z)$;\rule{0pt}{0.5cm}\\
  \bf{MCMC sampling}:\rule{0pt}{0.8cm}\\
   1. Initialise $\eta^{(0)}$, $z^{(0)}$ and $\gamma^{(0)}$;\rule{0pt}{0.5cm}\\
   2. For j=1:Num\rule{0pt}{0.5cm}\\
   Update $\eta^{(j)}$ according to $(\ref{flux_eta})$ $^*$ and update each component of it by Gibbs method as \rule{0pt}{0.5cm}\\
   discussed in Subsection \ref{rec-init};\rule{0pt}{0.5cm}\\
   Update $z^{(j)}$ basing on $(\ref{source})$ $^*$ by random walk MH algorithm;\rule{0pt}{0.5cm}\\
   Update $\gamma^{(j)}$ basing on $(\ref{lamda})$ by sampling from Gamma distribution directly.\rule{0pt}{0.5cm}\\
   end for\rule{0pt}{0.5cm}\\
     \hline
  $*$ means we change the equation by replacing $f_1(\eta)$ with $\mathbf{H}^M_2\eta$ and $f_2(z)$ with $\tilde G_N^M(z)$
   \end{tabular}
\end{table}

Measurement data are taken at at a uniform $9\times 9$ sensor network in space and time levels $[0.01:0.01:0.1]$. The forward model is solved at a uniform $80\times 80$ fine grid. With time step $\Delta t=0.001$, we generate  measurement data.  We use $\Delta t=0.002$ for solving the forward problem to avoid inverse crime. Measurement noise   $\sigma$ is set to be 0.005. We use a discretization of 20 grids in space and 11 basis functions in time to reconstruct the flux function, i.e., there are $n=220$ unknowns from the flux function. When the energy ratio is set as $e(n_0)=0.9$, we  $n_0=74$ random variables in truncated KLE to characterize the flux random field.
 This reduces the dimension of the parameters in flux and can speed up the MCMC sampling.
 We take $8\times8$ coarse grid and select  $10$ multiscale basis functions  each coarse neighborhood.  In the simulation, the total order of the gPC with respect to $z$ is $N=10$.

The ground truth parameter values  $z= (0.5,0.5)$ in this example.  We  run a chain of length $40000$ and take the last $50\%$  samples to compute the statistical quantities. The numerical results are shown in Figure \ref{flux} and \ref{po}. The relative $L_2$ error is about $5.68\%$  for the flux reconstruction. By Figure \ref{flux}, we find that  estimates at the initial time and the final time level are slightly poor. This is because the noise to signal ratio in the first few time steps is large, and the simulated data contains less information of the flux in the the last time levels. Figure \ref{po} shows the histograms, univariate and bivariate marginal posterior distributions of $z_1$ and $z_2$ , it can be seen from the estimated posterior distribution that the support of the posterior density is also narrower than the prior's.

%We notice that we use GMsFEM only to construct the surrogate model about $\eta$, and combine GMsFEM with LS-SCM to construct the surrogate model about $\chi$, it's necessary to consider if the result obtained by this scheme is consistent with our original target, hence we study the case where only the flux function or only the source location in system (5.22) needed to be recovered, and then to compare results obtained from recovering them simultaneously. As the first case, suppose the location is given, we use GMsFEM to construct the sensitivity matrix $H^M$, implies the surrogate model on $\chi$ has no effect on the whole surrogate model. Similarly, when the flux function is known, we use GMsFEM to solve the forward model and construct the surrogate model the whole surrogate model is stable. and the different sampling strategy

By using  the reduced order model, we have efficiently  recovered the boundary  flux and source location simultaneously.
 In a similar way, we can also use the method to recover the initial function, boundary condition  and the source location at the same time.
\begin{figure}[htbp]
 \centering
 \includegraphics[width=0.5\textwidth]{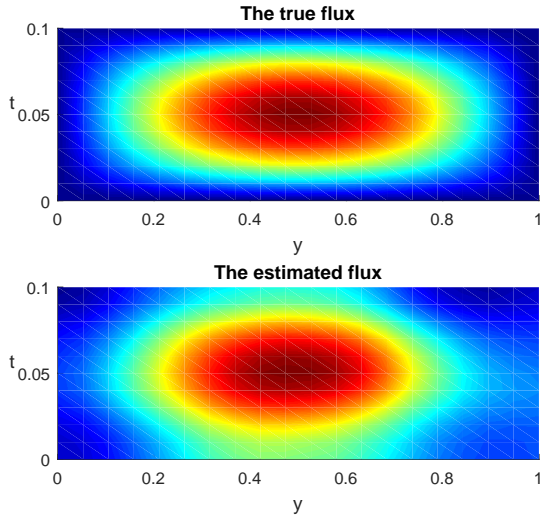}
 \caption{The true flux value vs the estimated flux, and the relative $L_2$ error is
5.68\%}\label{flux_source_thita}
\label{flux}
\end{figure}
%%%%%%%%%%%%%%%%%%%%%%%%%%%%%%%%%%

\begin{figure}[htbp]
 \centering
 \includegraphics[width=0.6\textwidth]{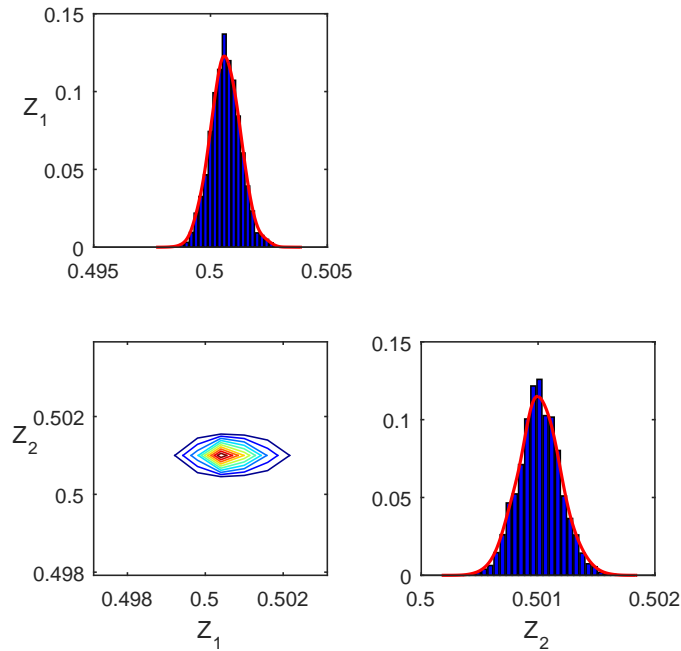}
 \caption{Histograms, univariate marginal and bivariate posterior distributions of $z_1$ and $z_2$, $M=10, N=10$.}
 \label{po}
\end{figure}

\section{Conclusion}

 The  paper has presented a multiscale model reduction method with application in Bayesian inverse problem for subsurface flow.   The reduced order model has accelerated  the MCMC sampling  under the framework of Bayesian inference. GMsFEM is used to construct sensitivity matrices for linear inverse problems.  For nonlinear inverse problems,  we have established the reduced order model by combing GMsFEM with LS-SCM.
  The  forward model is solved by GMsFEM over the support of the prior at the off-line stage. The presented strategy  leads to an accurate  approximation of the full-order forward model and gives a surrogate posterior density,
   which is easier to be evaluated than the original posterior.  Mathematical analysis is carried out for the approximation of reduced order method in the Bayesian inverse problem.
    We have numerically discussed the inverse problems in a confined subsurface flow model.
     The numerical examples confirms that the approximated posterior approximates the reference posterior very well by  using the multiscale model reduction method.

We have used  KLE to represent the unknown field in the inverse problems. The dimension reduction technique requires some degree of correlation or structure in the prior, one would ultimately like to find a basis emphasizing features of the unknown field that are most affected by the data \cite{tc14}. When constructing surrogate model using LS-SMC over support of the prior density, we find the support of the resultant posterior density is much narrower than the prior's.  Advanced  MCMC methods such as sequential Monte Carlo, two stage MC or multilevel MCMC can be used to generate some intermediate density, in which the data information can be incorporated \cite{jw11}. Then we can construct the reduced order  model over the support of the intermediate density. We may create a model order reduction method that incorporates the data information. For example, when constructing the coarse subspace, we use the data information for the mode direction, and hence the dimension of the subspace would decrease as the addition of the data information.  Further investigation of these issues is worth pursuing in the future.

\section*{Acknowledgments}
We acknowledge the support of Chinese NSF 11471107.


\begin{thebibliography}{99}
\bibitem{tc14}
{\sc T. Cui, J. Martin and Y. Marzouk},
Likelihood-informed dimension reduction for nonlinear inverse problems, Inverse Problems, 30 (2014), 114015.
\bibitem{mc15}
{\sc M. Chevreuil, R. Lebrun and A. Nouy},
A least-squares method for sparse low rank approximation of multivariate functions, SIAM/ASA Journal on Uncertainty Quantification, 3 (2015), pp. 897-921.
\bibitem{jc12}
{\sc J. A. Christen and C. Fox},
Markov chain Monte Carlo using an approximation, Journal of Computational and Graphical statistics, 14 (2005), pp. 795-810.
\bibitem{pc07}
{\sc P. Congdon},
Bayesian statistical modelling, John Wiley \& Sons, 2007.
\bibitem{ye05}
{\sc Y. Efendiev, A. Datta-Gupta and V. Ginting},
An efficient two-stage Markov chain Monte Carlo method for dynamic data integration, Water Resources Research, 411 (2005), 12423.
%\bibitem{ye}
%{\sc Y. Efendiev},
%Multiscale Model Reduction Techniques for Flows in High-contrast Heterogeneous Media and Applications.
\bibitem{ye13}
{\sc Y. Efendiev, J. Galvis and T. Hou},
Generalized multiscale finite element methods (GMsFEM), Journal of Computational Physics, 251 (2013), pp. 116-135.
\bibitem{ye11}
{\sc Y. Efendiev, J. Galvis and X. Wu},
Multiscale finite element methods for high-contrast problems using local spectral basis functions, Journal of Computational Physics, 230 (2011), pp. 937-955.
%\bibitem{ye06}
%{\sc Y. Efendiev, V. Ginting and T. Hou},
%Accurate multiscale finite element methods for two-phase flow simulations, Journal of Computational Physics, 220 (2006), pp. 155-174.
\bibitem{ye09}
{\sc Y. Efendiev and T. Hou},
Multiscale finite element methods: theory and applications, Springer Science \& Business Media, 2009.
\bibitem{hw00}
{\sc H. W. Engl,  M. Hanke and A. Neubauer},
Regularization of Inverse Problems, Springer Science \& Business Media, 2000.
\bibitem{mf10}
{\sc M. Frangos, Y. Marzouk and K. Willcox},
Surrogate and reduced-order modeling: A comparison of approaches for large-scale statistical inverse problems, Large-Scale Inverse Problems and Quantification of Uncertainty, John Wiley \& Sons,Ltd, 2010, pp. 123-149.
\bibitem{bg07}
{\sc B. Ganapathysubramanian and N. Zabaras},
Sparse grid collocation schemes for stochastic natural convection problems, Journal of Computational Physics, 225 (2007), pp. 652-685.
\bibitem{mh10}
{\sc M. Hinze and S. Volkwein},
Proper orthogonal decomposition surrogate models for nonlinear dynamical systems: Error estimates and suboptimal control, Dimension Reduction of Large-Scale Systems, Springer: Berlin \& Heidelberg, 2004, pp. 261-306.


\bibitem{hwy97}
{\sc T. Y. Hou and X. H. Wu},
A multiscale finite element method for elliptic problems in composite materials and porous media, J. Comput. Phys. 134 (1997),  pp. 169-189.


\bibitem{jegl07}
{\sc L. Jiang, Y. Efendiev, and V. Ginting},
  Multiscale methods for parabolic equations with continuum spatial scales, DCDS Series B, 8 (2007),  pp. 833-859.



\bibitem{jk06}
{\sc J. Kaipio and E. Somersalo},
Statistical and computational inverse problems, Springer Science \& Business Media, 2006.


\bibitem{llz09}
{\sc W. Li, Z. Lu and D. Zhang}, {\em Stochastic analysis of unsaturated flow with probabilistic collocation method}, Water Resources Research, 45 (2009), DOI: 10.1029/2008WR007530.


\bibitem{lj08}
{\sc J. Liu},
Monte Carlo strategies in scientific computing, Springer Science \& Business Media, 2008.
\bibitem{jm12}
{\sc J. Martin, L. C. Wilcox, C. Burstedde},
A stochastic Newton MCMC method for large-scale statistical inverse problems with application to seismic inversion, Siam Journal on Scientific Computing, 34 (2012), pp. A1460-A1487.
\bibitem{ym09}
{\sc Y. Marzouk and D. Xiu},
A Stochastic Collocation Approach to Bayesian Inference in Inverse Problems, Communications in Computational Physics, 6 (2009), pp. 826-847.
\bibitem{ymm09}
{\sc Y. Marzouk and  H. Najm},
Dimensionality reduction and polynomial chaos acceleration of Bayesian inference in inverse problems, Journal of Computational Physics, 228 (2009), pp. 1862-1902.
\bibitem{ym07}
{\sc Y. Marzouk, H. Najm and L. A. Rahn},
Stochastic spectral methods for efficient Bayesian solution of inverse problems, Journal of Computational Physics, 224 (2007), pp. 560-586.
%\bibitem{zo06}
%{\sc Z. Ostrowski},
%Application of proper orthogonal decomposition to the solution of inverse problems, Rozprawa doktorska, Gliwice: Pol. \'{S}l, 2006.
\bibitem{cr13}
{\sc C. Robert and G. Casella},
Monte Carlo statistical methods, Springer Science \& Business Media, 2013.
\bibitem{gr08}
{\sc G. Rozza, D. B. P. Huynh and A. T. Patera},
Reduced basis approximation and a posteriori error estimation for affinely parametrized elliptic coercive partial differential equations, Archives of Computational Methods in Engineering, 15 (2008), pp. 229-275.
\bibitem{rs13}
{\sc R. Sternfels and C. J. Earls},
Reduced-order model tracking and interpolation to solve PDE-based Bayesian inverse problems, Inverse Problems, 29 (2013), 075014.
\bibitem{as10}
{\sc A. M. Stuart},
Inverse problems: a Bayesian perspective, Acta Numerica, 19 (2010), pp. 451-559.
\bibitem{at05}
{\sc A. Tarantola},
Inverse problem theory and methods for model parameter estimation, Society for Industrial \& Applied Mathematics, 2005.
\bibitem{bw04}
{\sc B. Walsh},
Markov chain monte carlo and gibbs sampling, Notes, 91 (2004), pp. 497-537.
\bibitem{jw11}
{\sc J. Wan and N. Zabaras},
A Bayesian approach to multiscale inverse problems using the sequential Monte Carlo method, Inverse Problems, 27 (2011), 105004.
\bibitem{jw04}
{\sc J. Wang and N. Zabaras},
A Bayesian inference approach to the inverse heat conduction problem, International Journal of Heat and Mass Transfer, 47 (2004), pp. 3927-3941.
\bibitem{jw06}
{\sc J. Wang and N. Zabaras},
A Markov random field model of contamination source identification in porous media flow, International Journal of Heat and Mass Transfer, 49 (2006), pp. 939-950.
\bibitem{cw11}
{\sc C. Winton, J. Pettway and C. T. Kelley},
Application of proper orthogonal decomposition (POD) to inverse problems in saturated groundwater flow, Advances in Water Resources, 34 (2011), pp. 1519-1526.
\bibitem{dx07}
{\sc D. Xiu},
Efficient collocational approach for parametric uncertainty analysis, Communications in Computational Physics, 2 (2007), pp. 293-309.
\bibitem{dx10}
{\sc D. Xiu},
Numerical methods for stochastic computations: a spectral method approach, Princeton University Press, 2010.
\bibitem{ly15}
{\sc L. Yan and L. Guo},
Stochastic Collocation Algorithms Using $l_1$-Minimization for Bayesian Solution of Inverse Problems, SIAM Journal on Scientific Computing, 37 (2015), pp. A1410-A1435.
\bibitem{ly12}
{\sc L. Yan, L. Guo and D. Xiu},
Stochastic collocation algorithms using $l_1$-minimization, International Journal for Uncertainty Quantification, 2 (2012), pp. 279-293.



\end{thebibliography}
\end{document}